\documentclass{article}
\usepackage{fullpage}
\usepackage{amsthm}
\usepackage{amsmath}
\usepackage{amsfonts}
\usepackage{amssymb}
\usepackage{algorithm}
\usepackage{algorithmic}
\usepackage[shortlabels]{enumitem}
\usepackage{mathtools}
\usepackage{subfigure}
\usepackage{url}
\usepackage[titletoc]{appendix}

\def\noprint#1{}

\newtheorem{Str}{Implementation Strategy}[section]

\def\tto{\;{\lower 1pt \hbox{$\rightarrow$}}\kern -10pt
           \hbox{\raise 2.8pt \hbox{$\rightarrow$}}\;}

\newcommand{\lambdamin}{\lambda_{\mbox{\rm\scriptsize{min}}}}

\newcommand{\N}{\mathbb{N}}
\newcommand{\R}{\mathbb{R}}
\newcommand{\Sym}{\mathbb{S}}
\newcommand{\cS}{{\cal S}}

\newcommand{\eps}{\epsilon}
\newcommand{\epsg}{\epsilon_g}
\newcommand{\epsgbar}{\bar\epsilon_g}
\newcommand{\epsH}{\epsilon_H}

\newcommand{\flow}{f_{\mbox{\rm\scriptsize low}}}

\DeclareMathOperator{\lspan}{span}

\newcommand{\skCG}{s_k^{\rm CG}}

\newcommand{\capCG}{\mbox{\rm \texttt{capCG}}}

\newcommand{\tcO}{\tilde{\mathcal O}}

\DeclareRobustCommand{\Cmeo}{\mathcal{C}_{\mathrm{meo}}}

\def\revised#1{#1}
\def\revisedbis#1{#1}

\newcommand{\Kcal}{{\cal K}}

\newcommand{\Kcalis}{{\cal I}}
\newcommand{\Kcalbs}{{\cal B}}
\newcommand{\Scal}{{\cal S}}
\newcommand{\Ucal}{{\cal U}}
\newcommand{\lipg}{L_g}
\newcommand{\lipH}{L_H}
\newcommand{\cSleq}{\cS_{L}}
\newcommand{\cSgg}{\cS_{GG}}
\newcommand{\cSgleq}{\cS_{GL}}

\newcommand{\sout}{\mbox{\rm \texttt{outCG}}}

\newcommand{\OKBinf}{\mbox{\sc bnd-norm}}
\newcommand{\OKBneg}{\mbox{\sc bnd-neg}}
\newcommand{\OKIres}{\mbox{\sc int-res}}
\newcommand{\OKImax}{\mbox{\sc int-max}}
\newcommand{\itmax}{k_{\max}}
\newcommand{\Je}{J}

\newcommand{\pbpi}{\xi}
\newcommand{\true}{\textsc{true}}
\newcommand{\false}{\textsc{false}}

\usepackage{hyperref}
\hypersetup{colorlinks}
\usepackage[capitalize,nameinlink]{cleveref}
\crefname{equation}{}{}
\newtheorem{lemma}{Lemma}
\newtheorem{theorem}{Theorem}
\newtheorem{assumption}{Assumption}

\title{Trust-Region Newton-CG with Strong Second-Order Complexity Guarantees for Nonconvex Optimization\footnotemark[1]}

\author{
	Frank E. Curtis\footnotemark[2]
\and
	Daniel P. Robinson\footnotemark[2]
\and
	Cl\'ement W. Royer\footnotemark[3]
\and 
	Stephen J. Wright\footnotemark[4]
}

\begin{document}

\maketitle

\renewcommand{\thefootnote}{\fnsymbol{footnote}}
\footnotetext[1]{Version of \today.}
\footnotetext[2]{Department of Industrial and Systems Engineering, Lehigh University, 200 W. Packer Ave., Bethlehem, PA 18015-1582, USA. (\href{frank.e.curtis@lehigh.edu}{frank.e.curtis@lehigh.edu}, \href{daniel.p.robinson@lehigh.edu}{daniel.p.robinson@lehigh.edu}). Work of the first author was supported by DOE Award DE-SC0010615 and NSF Awards CCF-1740796 and CCF-1618717.}
\footnotetext[3]{LAMSADE, CNRS, Universit\'e Paris-Dauphine, Universit\'e PSL, 75016 PARIS, FRANCE. (\href{clement.royer@dauphine.psl.eu}{clement.royer@dauphine.psl.eu}). Work of this author was partially supported by Award N660011824020 from the DARPA Lagrange Program.}
\footnotetext[4]{Computer Sciences Department, University of Wisconsin, 1210 W. Dayton St., Madison, WI 53706, USA. (\href{swright@cs.wisc.edu}{swright@cs.wisc.edu}). Work of this author was supported by NSF Awards 1628384, 1634597, and 1740707; Subcontract 8F-30039 from Argonne National Laboratory; and Award N660011824020 from the DARPA Lagrange Program.}
\renewcommand{\thefootnote}{\arabic{footnote}}

\begin{abstract}
  Worst-case complexity guarantees for nonconvex optimization
  algorithms have been a topic of growing interest. Multiple
  frameworks that achieve the best known complexity bounds among a
  broad class of first- and second-order strategies have been
  proposed. These methods have often been designed primarily with
  complexity guarantees in mind and, as a result, represent a
  departure from the algorithms that have proved to be the most
  effective in practice.
  In this paper, we consider trust-region Newton methods, one of the
  most popular classes of algorithms for solving nonconvex
  optimization problems.  By introducing slight modifications to the
  original scheme, we obtain two methods---one based on exact
  subproblem solves and one exploiting inexact subproblem solves as in
  the popular ``trust-region Newton-Conjugate-Gradient'' (trust-region Newton-CG)
  method---with iteration and operation complexity bounds that match
  the best known bounds for the aforementioned class of first- and
  second-order methods. The resulting trust-region Newton-CG method also retains
  the attractive practical behavior of classical trust-region
  Newton-CG, which we demonstrate with numerical comparisons on a
  standard benchmark test set.
\end{abstract}



\section{Introduction} 
\label{sec:intro}


Consider the unconstrained optimization problem
\begin{equation} \label{eq:f}
\min_{x\in\R^n} \, f(x),
\end{equation}
where $f:\R^n \to \R$ is twice Lipschitz continuously differentiable
and possibly nonconvex.  We propose and analyze the complexity of two
trust-region algorithms for solving problem~\cref{eq:f}.  Our main
interest is in an algorithm that, for each subproblem, uses the
conjugate gradient (CG) method to minimize an exact second-order
Taylor series approximation of $f$ subject to a trust-region
constraint, as in so-called \emph{trust-region Newton-CG} methods.
Our complexity analysis for both methods is based on approximate
satisfaction of second-order necessary conditions for stationarity,
that is,
\begin{equation} \label{eq:2on}
  \nabla f(x)=0\ \ \text{and}\ \ \nabla^2 f(x) \;\; \mbox{positive semidefinite}.
\end{equation}
Specifically, given a pair of (small) real positive tolerances
$(\epsg,\epsH)$, our algorithms terminate when they find a point
$x^\eps$ such that
\begin{equation}\label{eq:2oneps}
  \|\nabla f(x^\eps)\| \le \epsg\ \ \text{and}\ \ \lambda_{\min}(\nabla^2 f(x^\eps)) \ge -\epsH,
\end{equation}
where $\lambda_{\min}(\cdot)$ denotes the minimum eigenvalue of its
symmetric matrix argument.  Such a point is said to be {\em
  $(\epsg,\epsH)$-stationary.}  By contrast, any point satisfying the
approximate first-order condition $\| \nabla f(x) \| \le \epsg$ is
called an {\em $\epsg$-stationary} point.

Recent interest in complexity bounds for nonconvex optimization stems
in part from applications in machine learning, where for certain
interesting classes of problems all local minima are global minima.
We have a particular interest in trust-region Newton-CG algorithms
since they have proved to be extremely effective in practice for a
wide range of large-scale applications.  We show in this paper that by
making fairly minor modifications to such an algorithm, we can equip
it with strong theoretical complexity properties without significantly \revised{degrading important performance measures such as the number of iterations, function evaluations, and gradient evaluations required until an $(\epsg,\epsH)$-stationary point is reached}.  This is in contrast to other recently proposed
schemes that achieve good complexity properties, but have not
demonstrated such good performance in practice \revised{against a state-of-the-art trust-region Newton-CG algorithm; see, e.g.,}  \cite{NAgarwal_ZAllenZhu_BBullins_EHazan_TMa_2017,YCarmon_JCDuchi_OHinder_ASidford_2018a}.

We prove results concerning both {\em iteration complexity} and {\em
  operation complexity}. The former refers to a bound on the number of
``outer'' iterations required to identify a point that satisfies
\eqref{eq:2oneps}. For the latter, we identify a \emph{unit operation}
and find a bound on the number of such operations required to find a
point satisfying \eqref{eq:2oneps}.  As in earlier works on Newton-CG
methods \cite{CWRoyer_SJWright_2018,CWRoyer_MONeill_SJWright_2019},
the unit operation is either a gradient evaluation or a Hessian-vector
multiplication.  In both types of complexity---iteration and
operation---we focus on the dependence of bounds on the tolerances
$\epsg$ and $\epsH$.

Our chief contribution is to show that a trust-region Newton-CG method
can be modified to have state-of-the-art operation complexity
properties for locating an $(\epsg,\epsg^{1/2})$-stationary point,
matching recent results for modified line search Newton methods, cubic
regularization methods, and other approaches based on applying accelerated
gradient to nonconvex functions (see Section~\ref{subsec:intro:liter}).
The setting $\epsH=\epsg^{1/2}$ is known to yield the lowest operation
complexity bounds for several classes of second-order algorithms.

\subsection{Outline}

We specify assumptions and notation used throughout the paper in
Section~\ref{subsec:intro:notat}, and discuss relevant literature
briefly in Section~\ref{subsec:intro:liter}. Section~\ref{sec:exact}
describes a trust-region Newton method in which we assume that the
subproblem is solved exactly at each iteration, and in which the
minimum eigenvalue of the Hessian is calculated as necessary to verify
the conditions \eqref{eq:2oneps}. We prove the iteration complexity of
this method, setting the stage for our investigation of a method using
inexact subproblem solves. In Section~\ref{sec:tcg}, we describe an
inexact implementation of the solution of the trust-region subproblem
by a conjugate gradient method, and find bounds on the number of
matrix-vector multiplications required for this method. We also
discuss the use of iterative methods to obtain approximations to the
minimum eigenvalue of the Hessian. Section~\ref{sec:inexact} describes
a trust-region Newton-CG method that incorporates the inexact solvers
of Section~\ref{sec:tcg}, and analyzes its iteration and operation
complexity properties.  We describe \revised{implementation challenges in 
Section~\ref{sec:practice}, then detail our} computational experiments in
Section~\ref{sec:numer}. \revised{Finally, we} make some concluding observations in
Section~\ref{sec:conc}.

\subsection{Assumptions and notation}
\label{subsec:intro:notat}

We write $\R{}$ for the set of real numbers (that is, scalars), $\R^n$
for the set of $n$-dimensional real vectors, $\R^{m \times n}$ for the
set of $m$-by-$n$-dimensional real matrices, $\Sym^n \subset \R^{n
  \times n}$ for the set of $n$-by-$n$-dimensional real symmetric
matrices, and $\N{}$ for the set of nonnegative integers.  For
$v\in\R^{n}$, we use $\|v\|$ to denote the $\ell_2$-norm of~$v$.
Given scalars $(a,b) \in \R \times \R$, we write $a\perp b$ to mean $ab
= 0$.

In reference to problem~\cref{eq:f}, we use $g := \nabla f : \R^{n}
\to \R{}$ and $H := \nabla^2 f : \R^{n} \to \Sym^n$ to denote the
gradient and Hessian functions of $f$, respectively.  For each
iteration $k \in \N{}\revised{= \{0,1,2,\dotsc\}}$ of an algorithm for solving \cref{eq:f}, we let
$x_k$ denote the $k$th solution estimate (that is, iterate) computed. For
brevity, we append $k \in \N{}$ as a subscript to a function to denote
its value at the $k$th iterate, e.g., $f_k := f(x_k)$, $g_k :=
g(x_k)$, and $H_k := H(x_k)$.  The subscript $j \in \N{}$ is similarly
used for the iterates of the subroutines used for computing search
directions for an algorithm for solving \cref{eq:f}.  Given $H_k \in
\Sym^n$, we let $\lambda_k := \lambda_{\min}(H_k)$ denote the minimum
eigenvalue of $H_k$ with respect to $\R$.

Given functions $\phi : \R \to \R$ and $\varphi : \R \to [0,\infty)$,
  we write $\phi(\cdot) = {\mathcal O}(\varphi(\cdot))$ to indicate
  that $|\phi(\cdot)| \leq C \varphi(\cdot)$ for some $C \in
  (0,\infty)$.  Similarly, we write $\phi(\cdot) =
  \tcO(\varphi(\cdot))$ to indicate that $|\phi(\cdot)| \leq C
  \varphi(\cdot) |\log^c(\cdot)|$ for some $C \in (0,\infty)$ and $c
  \in (0,\infty)$.  In this manner, one finds that ${\mathcal
      O}(\varphi(\cdot)\log^c(\cdot)) \equiv \tcO(\varphi(\cdot))$ for
    any $c \in (0,\infty)$.


      The following assumption on the objective function in
      \cref{eq:f} is made throughout.

\begin{assumption}\label{assum:f}
  The objective function value sequence $\{f_k\}$ is bounded below by
  $\flow \in \R{}$.  The sequence of line segments $\{[x_k,x_k+s_k]\}$
  lies in an open set 
  over which $f$ is twice continuously
  differentiable and the gradient and Hessian functions are Lipschitz
  continuous with constants $\lipg \in (0,\infty)$ and $\lipH \in
  (0,\infty)$, respectively.
\end{assumption}

The following bounds are implied by \cref{assum:f} (see e.g.,
\cite{nesterov1998introductory}):
\begin{subequations} \label{eq:fub}
  \begin{align}
    \label{eq:fub.2}
    f(x_k+s_k) - f_k - g_k^T s_k - \tfrac12 s_k^T H_k s_k  & \le 
     \tfrac{\lipH}{6} \|s_k\|^3 && \text{for all $k\in\N$,}\\
    \label{eq:fub.3}
    \|g(x_k+s_k) - g_k - H_k s_k\| 
    & \le \tfrac{\lipH}{2}\|s_k\|^2 && \text{for all $k \in \N$,}\\
    \label{eq:fub.4}
    \text{and} \ \ \|H_k \| &\le L_g && \text{for all $k \in \N$.}
    \end{align}
\end{subequations}

\subsection{Literature review}
\label{subsec:intro:liter}

Complexity results for smooth nonconvex optimization algorithms abound
in recent literature. We discuss these briefly and give some pointers
below, with a focus on methods with best known complexity.

Cubic regularization
\cite[Theorem~1]{YuNesterov_BTPolyak_2006} has iteration complexity
$\mathcal{O}(\epsg^{-3/2})$ to find an $(\epsg,\epsg^{1/2})$-stationary
point; see also \cite{CCartis_NIMGould_PhLToint_2011b,CCartis_NIMGould_PhLToint_2012a}. Algorithms that find such a point with {\em operation}
complexity $\tilde{\mathcal{O}}(\epsg^{-7/4})$, with high probability,
were proposed in
\cite{NAgarwal_ZAllenZhu_BBullins_EHazan_TMa_2017,YCarmon_JCDuchi_OHinder_ASidford_2018a}.
(The ``high probability'' is due to the use of randomized iterative
methods for calculating a minimum eigenvalue and/or solving a linear
system.) A method that \emph{deterministically} finds an
$\epsg$-stationary point in $\tcO(\epsg^{-7/4})$ gradient evaluations
was described in \cite{YCarmon_JCDuchi_OHinder_ASidford_2017a}.

Line search methods that make use of Newton-like steps, the conjugate
gradient method for inexactly solving linear systems, and randomized
Lanczos for calculating negative curvature directions are described in
\cite{CWRoyer_SJWright_2018,CWRoyer_MONeill_SJWright_2019}. These
methods also have operation complexity $\tcO(\epsg^{-7/4})$ to find a
$(\epsg,\epsg^{1/2})$-stationary point, with high probability. The
method in \cite{CWRoyer_MONeill_SJWright_2019} finds an
$\epsg$-stationary point deterministically in
$\tcO(\epsg^{-7/4})$ operations, showing that the conjugate gradient method on 
nonconvex quadratics shares properties with accelerated 
gradient on nonconvex problems as described 
in~\cite{YCarmon_JCDuchi_OHinder_ASidford_2017a}.


\paragraph{Trust-region methods}
An early result of \cite{SGratton_ASartenaer_PhLToint_2008} shows that
standard trust-region methods require $\mathcal{O}(\epsg^{-2})$
iterations to find an $\epsg$-stationary point; this complexity was
shown to be sharp in
\cite{CCartis_NIMGould_PhLToint_2011c}. 
A trust-region Newton method with iteration complexity of
$\mathcal{O}(\max\{\epsg^{-3/2},\epsH^{-3}\})$ for finding an
$(\epsg,\epsH)$-stationary point is described in
\cite{FECurtis_DPRobinson_MSamadi_2017}. This
complexity matches that of cubic regularization methods
\cite{YuNesterov_BTPolyak_2006,CCartis_NIMGould_PhLToint_2011b,CCartis_NIMGould_PhLToint_2012a}.

Another method that uses trust regions in conjunction with a cubic
model to find an $(\epsg,\epsH)$-stationary point with guaranteed
complexity appears in \cite{JMMartinez_MRaydan_2017}. This is not a
trust-region method in the conventional sense because it fixes the
trust-region radius at a constant value. Other methods that combine
trust-region and cubic-regularization techniques in search of good
complexity bounds are described in
\cite{FECurtis_DPRobinson_MSamadi_2018b,JPDussault_2018,JPDussault_DOrban_2015,EBergou_YDiouane_SGratton_2017,EBergou_YDiouane_SGratton_2018}.

\paragraph{Solving the trust-region subproblem}
Efficient solution of the trust-region subproblem is a core aspect of
both the theory and practice of trust-region methods. In the context
of this paper, such results are vital in turning an iteration
complexity bound into an operation complexity bound.  The fact that
the trust-region subproblem (with a potentially nonconvex objective)
can be solved efficiently remains surprising to many.  This is
especially true since it has some complicating features, particularly
the ``hard case'' in which, in iteration $k \in \N{}$, the gradient
$g_k$ is orthogonal to the eigenspace of the Hessian $H_k$
corresponding to its minimum eigenvalue.

Approaches for solving trust-region subproblems based on matrix factorizations are described in
\cite{JJMore_DCSorensen_1983}; see also
\cite[Chapter~4]{JNocedal_SJWright_2006}. For large-scale problems,
iterative techniques based on the conjugate gradient (CG)
algorithm~\cite{TSteihaug_1983,PhLToint_1981} and the Lanczos
method~\cite{NIMGould_SLucidi_MRoma_PhLToint_1999,SGNash_1984} have
been described in the literature. Convergence rates for the
method of \cite{NIMGould_SLucidi_MRoma_PhLToint_1999} are presented in
\cite{LHZhang_CShen_RCLi_2017}, though results are weaker in the hard
case.

Global convergence rates in terms of the objective function values for
the trust-region subproblem are a recent focus; see for example
\cite{EHazan_TKoren_2016}, wherein the authors use an SDP relaxation,
and \cite{JWang_YXia_2017}, wherein the authors apply an accelerated
gradient method to a convex reformulation of the trust-region
subproblem (which requires an estimate of the minimum eigenvalue of
the Hessian). Both solve the trust-region subproblem to within $\eps$
of the optimal subproblem objective value in $\tcO(\eps^{-1/2})$ time.

A recent method based on Krylov subspaces is presented in
\cite{YCarmon_JCDuchi_2018}.  This method circumvents the hard case by
its use of randomization. Subsequent work in
\cite{NIMGould_VSimoncini_2019} derives a convergence rate for the
norm of the residual vectors in the Krylov-subspace approach.

The hard case does not present a serious challenge to the main
algorithm that we propose (\cref{alg:inexact2}).  When it
occurs, either the conjugate gradient procedure (\cref{alg:etcg})
returns an acceptable trial step, or else the minimum-eigenvalue
procedure (\cref{alg:meo}) will be invoked to find a negative
curvature step.



\section{An exact trust-region Newton method}
\label{sec:exact}

In this section, we propose a trust-region Newton method that uses,
during each iteration, the \emph{exact} solution of a (regularized)
trust-region subproblem.  The algorithm is described in
Section~\ref{subsec:exact:algo} and its complexity guarantees are
analyzed in Section~\ref{subsec:exact:wcc}.  Our analysis of this
method sets the stage for our subsequent method that uses inexact
subproblem solutions.

\subsection{The algorithm}
\label{subsec:exact:algo}

Our trust-region Newton method with exact subproblem solves, which is
inspired in part by the line search method proposed in
\cite{CWRoyer_SJWright_2018}, is written as \cref{alg:exact}. Unlike a
traditional trust-region method, the second-order stationarity tolerance
$\epsH \in (0,\infty)$ is used to quantify a regularization of the quadratic model $m_k : \R^{n} \to \R{}$ of $f$ at $x_k$ used in
the subproblem, which is given by
\begin{equation} \label{eq:m}
  m_k(x) := f_k + g_k^T(x - x_k) + \tfrac12 (x - x_k)^T H_k (x - x_k);
\end{equation}
see \cref{eq:algos:exacttrsubpb}.  Our choice of regularization makes
for a relatively straightforward complexity analysis because it causes
the resulting trial step $s_k$ to satisfy certain desirable objective
function decrease properties.  Of course, a possible downside is that
the practical behavior of the method may be affected by the choice of
the stationarity tolerance~$\epsH$, which is not the case for a
traditional trust-region framework.  
In any case,
the remainder of \cref{alg:exact} is identical to a traditional
trust-region Newton method.

Before presenting our analysis of \cref{alg:exact}, we remark that \revisedbis{the sequence 
$\{\lambda_k\}$ of minimum eigenvalues of the Hessians $\{H_k\}$} does not influence the iterate sequence $\{x_k\}$.
The only use of these values is in the termination test in
line~\ref{step:terminate} to determine when an
$(\epsg,\epsH)$-stationary point has been found.

\begin{algorithm}[ht!]
\caption{Trust-Region Newton Method (exact version)}
\label{alg:exact}
\begin{algorithmic}[1]
\REQUIRE Tolerances $\epsg \in (0,\infty)$ and $\epsH \in (0,\infty)$; \revised{trust-region adjustment} parameters $\gamma_1 \in (0,1)$, $\gamma_2 \in [1,\infty)$, \revised{and $\psi \in (1/\gamma_2,1]$}; initial iterate $x_0 \in \R^n$; initial trust-region radius $\delta_0 \in (0,\infty)$; maximum trust-region radius $\delta_{\max} \in [\delta_0,\infty)$; and step acceptance parameter $\eta \in (0,1)$.
\smallskip
\hrule
\smallskip
\FOR{$k=0,1,2,\dotsc$}
\STATE Evaluate $g_k$ and $H_k$.
\STATE Initialize $\lambda_k \gets \infty$.
\IF{$\|g_k\|\le \epsg$}
\STATE Compute $\lambda_k \gets \lambda_{\min}(H_k)$.\label{step:lambda}
\IF{$\lambda_k \ge -\epsH$}\label{step:terminate}
\STATE \textbf{return} $x_k$ as an $(\epsg,\epsH)$-stationary point for problem~\cref{eq:f}.
\ENDIF
\ENDIF
\STATE Compute a trial step $s_k$ as a solution to the regularized trust-region subproblem
\begin{equation} \label{eq:algos:exacttrsubpb}
  \min_{s \in \R^n}\ m_k(x_k+s) + \tfrac{1}{2}\epsH \|s\|^2 \ \ 
  \mathrm{s.t.}\ \ \|s\| \le \delta_k.
\end{equation}
\STATE Compute the ratio of actual-to-predicted reduction in $f$, defined as
\begin{equation}\label{eq.rho} 
  \rho_k \gets \frac{f_k-f(x_k+s_k)}{m_k(x_k)-m_k(x_k+s_k)}.
\end{equation}
\IF{$\rho_k \ge \eta$}
\STATE Set $x_{k+1} \gets x_k+s_k$.
\revised{\IF{$\|s_k\| \ge \psi \delta_k$}
\STATE Set $\delta_{k+1} \gets \min\left\{\gamma_2\delta_k,\delta_{\max}\right\}$.
\ELSE \STATE Set $\delta_{k+1} \gets \delta_k$.
\ENDIF}
\ELSE
\STATE Set $x_{k+1} \gets x_k$ and $\delta_{k+1} \gets \gamma_1 \revised{ \|s_k\|}$.
\ENDIF
\ENDFOR
\end{algorithmic}
\end{algorithm}

\subsection{Iteration complexity}
\label{subsec:exact:wcc}

We show that \cref{alg:exact} reaches an $(\epsg,\epsH)$-stationary
point in a number of iterations that is bounded by a function of
$\epsg$ and $\epsH$.  To this end, let us define the set of iteration
numbers
\[
  \Kcal := \{k \in \mathbb{N} : \text{iteration $k$ is completed without algorithm termination}\}
\]
along with the subsets
\[
    \Kcalis := \{k \in \Kcal : \|s_k\| < \delta_k \} 
    \quad \text{and} \quad 
    \Kcalbs := \{k \in \Kcal : \|s_k\| = \delta_k \}
\]
and
\[
\Scal := \{k \in \Kcal : \rho_k \geq \eta\} \quad \text{and} \quad
\Ucal := \{k \in \Kcal : \rho_k < \eta\}.
\]
The pairs $(\Kcalis,\Kcalbs)$ and $(\Scal,\Ucal)$ are each
partitions of $\Kcal$.
The iterations with $k \in \Kcalis$ are those with~$s_k$ in the
\emph{interior} of the trust region, and those with $k\in \Kcalbs$ are
those with~$s_k$ on the \emph{boundary} of the trust region.  The
iterations with $k \in \Scal$ are called the \emph{successful}
iterations and those with $k \in \Ucal$ are called the
\emph{unsuccessful} iterations.  Due to the termination conditions in
line~\ref{step:terminate}, it follows for \cref{alg:exact} that
\begin{equation}\label{set:K}
  \Kcal = \{k \in \mathbb{N} : \text{iteration $k$ is reached and either $\|g_k\| > \epsg$ or $\lambda_k < -\epsH$}\}.
\end{equation}
It follows that, for a run of \cref{alg:exact}, the cardinalities of
all of the index sets $\Kcal$, $\Kcalis$, $\Kcalbs$, $\Scal$, and
$\Ucal$ are functions of the tolerance parameters $\epsg$ and $\epsH$.

Since $s_k$ is computed as the global solution of the trust-region
subproblem~\cref{eq:algos:exacttrsubpb}, it is well
known~\cite{JJMore_DCSorensen_1983,JNocedal_SJWright_2006} that there
exists a scalar Lagrange multiplier $\mu_k$ such that
\begin{subequations} \label{eq.tr_opt}
  \begin{align}
    g_k + (H_k + \epsH I + \mu_k I)s_k &= 0, \label{eq.tr_opt_1} \\
    H_k + \epsH I + \mu_k I &\succeq 0, \label{eq.tr_opt_2}\\
    \text{and} \ \ 0 \leq \mu_k \perp (\delta_k - \|s_k\|) &\geq 0 \label{eq.tr_opt_3}.
  \end{align}
\end{subequations}

Our first result is a lower bound on the model reduction achieved by a trial step.

\begin{lemma}\label{lem:exact.modeldec}
  For all $k \in \Kcal$, the model reduction satisfies
  \begin{equation}\label{eq:exact:decreasemodel}
    m_k(x_k) - m_k(x_k + s_k) \geq \tfrac{1}{2}\epsH \|s_k\|^2.
  \end{equation}
\end{lemma}
\begin{proof}
  The definition of $m_k$ in~\cref{eq:m} and the optimality conditions
  in~\cref{eq.tr_opt} give
  \begin{eqnarray*}
    m_k(x_k) - m_k(x_k+s_k) &= &-g_k^T s_k - \tfrac{1}{2} s_k^T H_k s_k \\
    &= & s_k^T (H_k + \epsH I + \mu_k I) s_k - \tfrac{1}{2}s_k^T H_k s_k \\ 
    &= & \tfrac{1}{2} s_k^T (H_k + \epsH I + \mu_k I) s_k + \tfrac{1}{2}\epsH \|s_k\|^2 
    +\tfrac{1}{2}\mu_k \|s_k\|^2 \\
    &\ge &	\tfrac{1}{2}\epsH \|s_k\|^2,
  \end{eqnarray*}
  as desired.
\end{proof}

Next, we show that all sufficiently small trial steps 
lead to successful iterations.

\begin{lemma}\label{lem:exact:delta-lb}
  For all $k \in \Kcal$, 
  \revised{if $k \in \Ucal$, then $\delta_k > 3(1-\eta) \epsH/\lipH$.}
  Hence, by the trust-region radius update
  procedure, it follows that
  \begin{equation*}
    \delta_k \geq \delta_{\min}:= \min\left\{\delta_0, \left(\tfrac{3 \gamma_1 (1-\eta)}{\lipH} \right)\epsH \right\} \in (0,\infty)\ \ \text{for all}\ \ k \in \Kcal.
  \end{equation*}
\end{lemma}
\begin{proof}
  \revised{We begin by proving the first statement of the lemma.} 
  To that end, suppose that
  $k\in\Ucal$ (meaning that $\rho_k <
  \eta$), which from the definition of $\rho_k$ means that
  \begin{equation}\label{rho-violated}
    \eta\left(m_k(x_k+s_k)-m_k(x_k)\right) < f(x_k+s_k)-f_k.
  \end{equation}
  Combining~\cref{rho-violated} with \cref{eq:fub.2}, \cref{eq:algos:exacttrsubpb}, 
  \cref{lem:exact.modeldec}, and \cref{eq:m} leads to 
  \begin{eqnarray*}
    &\phantom{\Rightarrow}
    &\eta \left(m_k(x_k+s_k)-m_k(x_k)\right) < g_k^T s_k + 
    \tfrac{1}{2}s_k^T H_k s_k + \tfrac{\lipH}{6}\|s_k\|^3 \\
    &\implies 
    &(\eta-1)\left(m_k(x_k+s_k)-m_k(x_k)\right) < \tfrac{\lipH}{6}\|s_k\|^3 \\
    &\implies 
    &\tfrac{1-\eta}{2}\epsH\|s_k\|^2< \tfrac{\lipH}{6}\revised{\|s_k\|^3}\\
    &\implies &\tfrac{3(1-\eta)}{\lipH}\epsH < \revised{\|s_k\|}.
  \end{eqnarray*}
  We have shown that $k\in\Ucal$ implies $\delta_k \revised{\ge \|s_k\|} >
  3(1-\eta)\epsH/\lipH$, \revised{as desired. Using a contraposition argument, we also 
  have that $\|s_k\| \leq 3(1-\eta)\epsH/\lipH$} implies $k \in \Scal$.
  Combining this with the trust-region radius update procedure and
  accounting for the initial radius $\delta_0$ completes the proof.
\end{proof}

We now establish that each successful step guarantees that a certain amount
of decrease in the objective function value is achieved.

\begin{lemma} \label{lem:exact.actualdec} 
  The following hold for all successful iterations:
  \begin{enumerate}[(i)]  
    \item If $k \in \Kcalbs \cap \Scal$, then 
      \begin{equation*} 
        f_k - f_{k+1} \geq \tfrac{\eta}{2}\epsH\delta_k^2.
      \end{equation*}
    \item If $k \in \Kcalis \cap \Scal$, then
       \begin{equation*} 
        f_k - f_{k+1} \ge  
        \tfrac{\eta}{2(1+2\lipH)} \,
        \min\left\{ \|g_{k+1}\|^2 \epsH^{-1}, 
        \epsH^3 \right\}.
      \end{equation*}
  \end{enumerate}
\end{lemma}
\begin{proof}
Part (i) follows from \cref{lem:exact.modeldec} and the definition of $\Kcalbs$, which imply that
  \begin{equation*}
    f_k - f_{k+1} \ge 
    \eta\left(m_k(x_k)-m_k(x_k+s_k)\right) 
    \geq \tfrac{\eta}{2}\epsH\|s_k\|^2 
    = \tfrac{\eta}{2}\epsH\delta_k^2.
  \end{equation*}

	For part (ii), from $k\in\Kcalis$ we know that $\|s_k\| <
        \delta_k$.  This fact along with \cref{eq.tr_opt_3}
        and~\cref{eq.tr_opt_1} imply that $\mu_k=0$ and $g_k + (H_k +
        \epsH I) s_k = 0$.  Now, with \cref{eq:fub.3}, we have
  \begin{align*}
    \|g_{k+1}\| &= \|g_{k+1}-g_k - (H_k + \epsH I) s_k \| \\
    &\le \|g_{k+1}- g_k - H_k s_k\| + \epsH \|s_k\| \le \tfrac{\lipH}{2}\|s_k\|^2 + \epsH \|s_k\|,
  \end{align*}
  which after rearrangement yields
  \[
  \tfrac{\lipH}{2}\|s_k\|^2+ \epsH \|s_k\| - \|g_{k+1}\| \ge 0.
  \]  
  Treating the left-hand side as a quadratic scalar function of
  $\|s_k\|$ implies that
  \[
  \|s_k\| 
  \ge \tfrac{-\epsH + \sqrt{\epsH^2 + 2 \lipH \|g_{k+1}\|}}{\lipH} 
  = \tfrac{-1+\sqrt{1 + 2 \lipH  \|g_{k+1}\|\epsH^{-2}}}{\lipH} \, \epsH.
  \]
  \revised{
    To put this lower bound into a slightly more useful form, we use
  \cite[Lemma~17 in Appendix A]{CWRoyer_SJWright_2018}, which states
  that for scalars $(a,b,t) \in (0,\infty) \times (0,\infty) \times
  [0,\infty)$, we have
    \begin{equation} \label{eq:ws8}
    -a + \sqrt{a^2 + bt} \geq (-a + \sqrt{a^2 + b}) \min\{t,1\}.
    \end{equation}
    By setting $a=1$, $b=2 \lipH$, and $t=\|g_{k+1}\|\epsH^{-2}$, we obtain
    }
    \begin{align*}
      \|s_k\|
      &\geq \left(\tfrac{-1+\sqrt{1+2 \lipH}}{\lipH}\right)
      \min\left\{ \|g_{k+1}\|\epsH^{-2},1\right\}\epsH \nonumber \\
      &= \left(\tfrac{2 \lipH}{\lipH(1+\sqrt{1+2 \lipH})}\right)
      \min\left\{ \|g_{k+1}\|\epsH^{-1},\epsH\right\}  \\
      & \ge \left(\tfrac{1}{\sqrt{1+2\lipH}}\right)
      \min\left\{ \|g_{k+1}\|\epsH^{-1},\epsH\right\}.
    \end{align*}
    Using this inequality in conjunction with $k\in\Scal$ and 
    \cref{lem:exact.modeldec} proves that
    \begin{align*} 
      f_k - f(x_k+s_k) 
      &\geq \eta\left(m_k(x_k)-m_k(x_k+s_k)\right)
      \geq \tfrac{\eta}{2}\epsH\|s_k\|^2 \nonumber \\
      &\geq \tfrac{\eta}{2(1+2\lipH)}
      \min\left\{\|g_{k+1}\|^2\epsH^{-1},\epsH^3\right\},
    \end{align*}
    which completes the proof for part (ii). 
\end{proof}

We now bound the number of successful iterations before termination.

\begin{lemma} \label{lem:exact.wccsuccc}
  The number of successful iterations performed by \cref{alg:exact} 
  before an $(\epsg,\epsH)$-stationary point is reached satisfies
  \begin{equation} \label{eq:exactwccitssucc}
      |\Scal| \leq \left\lfloor\mathcal{C}_\cS\max\{\epsH^{-1},
  \epsg^{-2}\epsH,\epsH^{-3}\} \right\rfloor + 1,
  \end{equation}
  where
  \begin{equation} \label{eq:exactwccitsCc}
    \mathcal{C_\cS} := \tfrac{4(f_0 - \flow)}{\eta}
    \max\left\{\tfrac{1}{\delta_0^2},\tfrac{\lipH^2}{9\gamma_1^2(1-\eta)^2},1+2\lipH\right\}.
  \end{equation}
\end{lemma}
\begin{proof}
  The successful iterations may be written as 
  $\cS =  \cSleq \cup \cSgg \cup \cSgleq$ where
  \begin{align*}
    \cSleq &:= \{ k\in\cS: \|g_k\| \leq \epsg\}, \\
    \cSgg &:= \{ k\in\cS: \|g_k\| > \epsg \ 
    \text{and} \ \|g_{k+1}\| > \epsg \}, \\
    \text{and} \ \ 
    \cSgleq &:= \{ k\in\cS: \|g_k\|>\epsg \ 
    \text{and} \ \|g_{k+1}\| \leq \epsg \}.
  \end{align*}
  
  We first bound $|\cSleq \cup \cSgg|$, for which we will make use of the constant
    \begin{equation}\label{def:c}
  		c:= \tfrac{\eta}{2}\min\left\{
  		\delta_0^2,
		\tfrac{9\gamma_1^2(1-\eta)^2}{\lipH^2}, 
  		\tfrac{1}{1+2\lipH}\right\}.
  \end{equation}
  For $k \in \cSleq$, the fact
  that the algorithm has not yet terminated implies (see \cref{set:K})
  that $\lambda_{k} < -\epsH$.  By \cref{eq.tr_opt}, it
    follows that $\mu_k > 0$ and $\|s_k\|=\delta_k$, and thus $k \in
    \Kcalbs$.  Thus, for $k\in\cSleq$, \cref{lem:exact.actualdec}(i), \cref{lem:exact:delta-lb}, and \eqref{def:c} imply that
  \begin{align}
    f_k-f_{k+1} 
    &\geq \tfrac{\eta}{2}\epsH\delta_k^2  
    \geq \tfrac{\eta}{2}\min\left\{\delta_0^2\epsH,
    \tfrac{9\gamma_1^2(1-\eta)^2}{\lipH^2} \epsH^3\right\} \ge
     c \min \left\{ \epsH, \epsH^3 \right\}.\label{dec:S1}
  \end{align}
  Now consider $k\in\cSgg$.  Since in this case either of the cases in
  \cref{lem:exact.actualdec} may apply, one can only conclude
  that, for each $k\in\cSgg$, the following bound holds:
  \begin{align*}
    f_k-f_{k+1} 
    &\geq \tfrac{\eta}{2}\min\left\{\delta_k^2\epsH, 
    \left(\tfrac{\|g_{k+1}\|^2}{1+2\lipH}\right)\epsH^{-1}, 
    \left(\tfrac{1}{1+2\lipH}\right)\epsH^3 \right\}.
  \end{align*}
  Combining this with the definition of $\cSgg$, 
  the lower bound on $\delta_k$ in~\cref{lem:exact:delta-lb}, and the
  definition  of $c$ in \cref{def:c}, it follows that
  \begin{equation}\label{dec:S2}
  		f_{k}-f_{k+1} 
  		\geq  c\min\left\{\epsH,\epsg^2\epsH^{-1},\epsH^3\right\}.
  \end{equation}    
  To bound $|\cSleq \cup \cSgg|$, we sum the objective function
  decreases obtained over all such iterations,
  which with~\cref{assum:f} and the monotonicity of $\{f_k\}$ gives
  $$
    f_0 - \flow 
    \ge \sum_{k\in\Kcal} (f_k - f_{k+1}) 
    \ge \sum_{k \in \cSleq} (f_k - f_{k+1}) +
    \sum_{k \in \cSgg} (f_k - f_{k+1}).
  $$  
  Combining this inequality with~\cref{dec:S1} 
  and~\cref{dec:S2} shows that
  \begin{align*}
		f_0 - \flow 
		&\geq \sum_{k \in \cSleq} c\min\{\epsH,\epsH^3\} + 
    \sum_{k \in \cSgg} c\min\{\epsH,\epsg^2\epsH^{-1},\epsH^3 \} \\
		&\geq c(|\cSleq| + |\cSgg|) \min\{\epsH,\epsg^2\epsH^{-1},\epsH^3\},
  \end{align*}
  from which it follows that
  \begin{equation} \label{bounds-first2}
   |\cSleq| +   |\cSgg| \leq \left(\tfrac{f_0 - \flow}{c}\right)\max\{\epsH^{-1},\epsg^{-2}\epsH,\epsH^{-3}\}.
  \end{equation}
  
  Next, let us consider the set $\cSgleq$. Since $k\in\cSgleq$ means
  $\|g_{k+1}\|\leq \revised{\epsg}$, the index corresponding to the next
  successful iteration (if one exists) must be an element of the index
  set~$\cSleq$.  This implies that $ |\cSgleq| \leq |\cSleq| + 1 $,
  where the $1$ accounts for the possibility that the last successful
  iteration (prior to termination) has an index
  in~$\cSgleq$. Combining this bound with~\cref{bounds-first2} yields
  \begin{equation*}
    |\cS|
    = |\cSleq| + |\cSgg| + |\cSgleq|
    \leq \tfrac{2(f_0-\flow)}{c}\max\{\epsH^{-1},\epsg^{-2}\epsH,\epsH^{-3}\} + 1,
  \end{equation*}
  which completes the proof when we substitute for $c$ from \cref{def:c}.
\end{proof}


We now bound the number of unsuccessful iterations.

\begin{lemma} \label{lem:exact.wccunsuccsucc}
The number of unsuccessful iterations that occur before an
$(\epsg,\epsH)$-stationary point is reached \revised{is either zero or
  else} satisfies
  \begin{equation} \label{eq:exact.wccunsuccsucc}
    |\Ucal| \leq 
    \left\lfloor 1+
    \log_{\gamma_1}\left(\tfrac{3(1-\eta)}{\lipH \delta_{\max}}\right) + 
    \log_{\gamma_1}\left(\epsH\right)
    \right\rfloor |\Scal|.
  \end{equation}
\end{lemma}
\begin{proof}
\revised{If the number of successful iterations is zero, then the initial point must be $(\epsg,\epsH)$-stationary and there are no unsuccessful iterations. Hence, let us proceed under the assumption that $|\Scal| \geq 1$.}
Let us denote the successful iteration indices as
$\{k_1,\dots,k_{|\Scal|}\} := \Scal$.  \revised{If the number of unsuccessful iterations is zero, then there is nothing left to prove, so we may proceed under the assumption that there is at least one unsuccessful iteration.  Thus, we may consider arbitrary $i \in
  \{1,\dots,|\Scal|-1\}$ such that $k_{i+1} - k_i - 1 \ge 1$, that is,
  there is at least one unsuccessful iteration between iterations
  $k_i$ and $k_{i+1}$. We seek a bound on $k_{i+1} - k_i - 1$.} From the update formulas for the trust-region radius,
one finds for all \revised{unsuccessful iteration indices} $l \in
\{k_i+1,\dots,k_{i+1}-1\}$ that \revised{$\delta_l = \gamma_1
  \|s_{l-1}\| \le \gamma_1 \delta_{l-1}$, so}
\begin{equation} \label{eq:exact.wccunsuccdeltarule}
  \delta_l \le 
  \min\{\gamma_2\delta_{k_i},\delta_{\max}\} \gamma_1^{l-k_i-1} \le \delta_{\max} \gamma_1^{l-k_i-1}.
\end{equation}
\revised{Moreover, for any unsuccessful iteration index $l \in
  \{k_i+1,\dots,k_{i+1}-1\}$ we have from \cref{lem:exact:delta-lb}}
  that $\delta_l >
3(1-\eta)\epsH/\lipH$.  Thus, for such $l$,
\cref{eq:exact.wccunsuccdeltarule} implies that
\begin{equation*}
  \tfrac{3(1-\eta)}{\lipH}\epsH
  <  \delta_{\max}\gamma_1^{l-k_i-1}
  \implies
  l-k_i-1 \le \log_{\gamma_1}\left(
  \tfrac{3(1-\eta)\epsH}{\lipH\delta_{\max}}\right).
\end{equation*}
Consequently, using the specific choice $l = k_{i+1}-1$, one finds that
\[
k_{i+1}-k_i-1 \leq 1 + 
\log_{\gamma_1}\left(\tfrac{3(1-\eta)}{\lipH\delta_{\max}}\right) +
\log_{\gamma_1}(\epsH),
\]
\revised{
and because the left-hand side is an integer, we have
\begin{equation}\label{eq:exact.wccunsuccboundiip1}
  k_{i+1}-k_i-1 \leq \left\lfloor 1 + 
  \log_{\gamma_1}\left(\tfrac{3(1-\eta)}{\lipH\delta_{\max}}\right) +
  \log_{\gamma_1}(\epsH) \right\rfloor.
\end{equation}
Since $i$ was chosen arbitrarily such that $k_{i+1}-k_i-1
\ge 1$, the right-hand side in
\eqref{eq:exact.wccunsuccboundiip1} is at least 1 if there are {\em
  any} unsuccessful iterations between iteration $k_1$ and
$k_{|\Scal|}$.}

\revised{Consider now the first successful iteration $k_1$.  We seek a bound on the number of unsuccessful iterations prior to iteration $k_1$.  If there are no such unsuccessful iterations, then there is nothing left to prove; hence, we may assume $k_1 \ge 1$.}  We have that $\delta_{k_1} \revised{\le}
\gamma_1^{k_1} \delta_0 \le \gamma_1^{k_1} \delta_{\max}$, and
\revised{thus from Lemma~\ref{lem:exact:delta-lb} it follows that
  \[
  \min \left\{ \delta_0, \tfrac{3\gamma_1 (1-\eta)}{\lipH} \epsH
  \right\} \le \delta_{k_1} \le \gamma_1^{k_1} \delta_0 \le
  \gamma_1^{k_1} \delta_{\max}.
  \]
From the first two of these inequalities and the facts that $\gamma_1 \in (0,1)$ and $k_1 \geq 1$, the ``$\min$'' on the left-hand side is {\em
  not} achieved by $\delta_0$, so we have
\[
\tfrac{3\gamma_1 (1-\eta)}{\lipH} \epsH \le 
\gamma_1^{k_1} \delta_{\max},
\]
which, when we take into account that $k_1$ is an integer, leads to
\begin{equation} \label{eq:xj3}
k_1 \le \left\lfloor 1 + \log_{\gamma_1}\left(\tfrac{3(1-\eta)}{\lipH\delta_{\max}}\right) +
  \log_{\gamma_1}(\epsH) \right\rfloor.
\end{equation}
Under our assumption that $k_1 \ge 1$, the
right-hand side of \eqref{eq:xj3} is at least $1$.}

Since the iteration immediately prior to termination is $k_{|\Scal|}$
(except in the trivial case in which termination occurs at \revised{the initial point}), we have
\begin{equation} \label{eq:xj4}
| \Ucal| = k_1 + \sum_{i=1}^{|\Scal|-1} (k_{i+1} - k_i - 1).
\end{equation}
\revised{If $\Ucal \neq \emptyset$, we have that $k_1 \ge 1$ and/or at
  least one of the terms in the summation is at least $1$. We can in
  this case bound {\em every} term on the right-hand side of
  \eqref{eq:xj4} by the right-hand sides of
  \cref{eq:exact.wccunsuccboundiip1} and \cref{eq:xj3} to deduce the
  result.}
\end{proof}

The main result for iteration complexity of Algorithm~\ref{alg:exact}
may now be proved.

\begin{theorem} \label{theo:exact.wccorder}
  Under \cref{assum:f}, the number of successful iterations $($and
  objective gradient and Hessian evaluations$)$ performed by
  \cref{alg:exact} before an $(\epsg,\epsH)$-stationary point is
  obtained satisfies
  \begin{equation} \label{eq:exact.wccordersuccits}
    |\Scal| = \mathcal{O} \left( \max\{\epsH^{-3},\epsH^{-1},\epsg^{-2}\epsH\}\right)
  \end{equation}
  and the total number of iterations $($and objective function
  evaluations$)$ performed before such a point is obtained
  satisfies
  \begin{equation} \label{eq:exact.wccorderits}
    \begin{aligned}
      |\Kcal| 
      &= \mathcal{O}\left(
      \log_{1/\gamma_1}(\epsH^{-1}) \max\{\epsH^{-3},\epsH^{-1},\epsg^{-2}\epsH\}
      \right).
    \end{aligned}
  \end{equation}
\end{theorem}
\begin{proof}
  Formula~\cref{eq:exact.wccordersuccits} follows
  from~\cref{lem:exact.wccsuccc}.
  Formula~\cref{eq:exact.wccorderits} follows from
  \cref{lem:exact.wccsuccc}, \cref{lem:exact.wccunsuccsucc}, and the
  fact that $\log_{\gamma_1}(\epsH) =
  \log_{1/\gamma_1}(\epsH^{-1})$. 
\end{proof}

If one chooses $\epsH=\epsg^{1/2}$ in~\cref{eq:exact.wccordersuccits}
and~\cref{eq:exact.wccorderits} as well as any positive scalar
$\epsgbar\in\R$, then \cref{theo:exact.wccorder} implies that, for all
$\epsg\in(0,\epsgbar]$, one has
\[
|\Scal| = \mathcal{O}\left(\epsg^{-3/2}\right)
\ \ \text{and} \ \
|\Kcal| = \mathcal{O}\left(\epsg^{-3/2}\log_{1/\gamma_1}\left(\epsg^{-1/2}\right)\right) 
= \tcO\left(\epsg^{-3/2}\right)
\]
for the numbers of successful and total iterations, respectively.
These correspond to the results obtained for the line search method
in~\cite[Theorem~5,~Theorem~6]{CWRoyer_SJWright_2018}).

\section{Iterative methods for solving the subproblems inexactly}
\label{sec:tcg}



This section describes the algorithms needed to develop an
\emph{inexact} trust-region Newton method, which will be presented and
analyzed in Section~\ref{sec:inexact}. A truncated CG method for
computing directions of descent is discussed in
Section~\ref{subsec:tcg:classic} and an iterative algorithm for
computing directions of negative curvature is described in
Section~\ref{subsec:inexact:meo}.

\subsection{A truncated CG method}\label{subsec:tcg:classic}

We propose \cref{alg:etcg} as an appropriate iterative method for
approximately solving the trust-region subproblem
\begin{equation} \label{eq:inexact1:trsubpb}
  \min_{s \in \R^n}  \ g^T s + \tfrac{1}{2}s^T (H + 2\eps I) s 
  \ \ \mathrm{s.t.} \ \
  \|s\| \le \delta, 
\end{equation}
where $g \in \R^n$ is assumed to be a non-zero vector, $H \in \Sym^n$
is possibly indefinite, $\eps \in (0,\infty)$ plays the role of a
regularization parameter, and $\delta \in (0,\infty)$. \cref{alg:etcg}
is based on the CG method and builds on the Steihaug-Toint
approach~\cite{TSteihaug_1983,PhLToint_1981}. (The factor of 2 in the
regularization term in~\cref{eq:inexact1:trsubpb} is intentional.  For
consistency in the termination condition, the inexact trust-region
Newton method in Section~\ref{sec:inexact} employs a larger
regularization term than the exact method analyzed in
Section~\ref{sec:exact}.)

For the most part, \cref{alg:etcg} is identical to traditional
truncated CG.  For example, termination occurs in
line~\ref{step:return.outside} when the next CG iterate $y_{j+1}$
would lie outside the trust region, and we return $s$ as the largest
feasible step on the line segment connecting $y_j$ to $y_{j+1}$. In
this situation, we also set $\sout \gets \OKBinf$ to indicate that $s$
lies on the boundary of the trust-region constraint.

However, there are three key differences between \cref{alg:etcg} and
truncated~CG. First, the residual termination criterion in
line~\ref{step:res} enforces the condition
\begin{equation}\label{eq:cvcrittcg}
	\|(H + 2\eps I) s + g\| \; \le \; \tfrac{\zeta}{2} \min\{\|g\|,\eps\|s\|\},
\end{equation}
which is stronger than the condition traditionally used in truncated
CG (which typically has $\tfrac{\zeta}{2}\|g\|$ for the right-hand
side) and incorporates a criterion typical of Newton-type methods with
optimal
complexity~\cite{CCartis_NIMGould_PhLToint_2019,CWRoyer_SJWright_2018}
(which use $\eps\|s\|$ for the right-hand side).  If this criterion is
satisfied, then we return the current CG iterate as the step~(that is,
we set $s \gets y_{j+1}$) and indicate that $s$ lies in the interior
of the trust region and satisfies~the residual condition
\cref{eq:cvcrittcg} by setting $\sout \gets \OKIres$.


Second, traditional truncated CG terminates if a direction of
nonpositive curvature is encountered.  Line~\ref{step:negcurve_test}
of \cref{alg:etcg} triggers termination if a direction with curvature
less than or equal to $\eps$ is found for $H+2\eps I$, since this
condition implies that the curvature of $H$ along the same direction
is less than \revised{or equal to} $-\eps$. In this case, we return a step~$s$ obtained by
moving along the direction of negative curvature to the boundary of
the trust-region constraint, and return $\sout \gets \OKBneg$ to
indicate that $s$ lies on the boundary because a direction of negative
curvature was computed.

Third, unlike traditional truncated CG, which (in exact arithmetic)
requires up to a maximum of $\itmax = n$ iterations,
line~\ref{step:cap} allows for an alternative iteration limit to be
imposed. Regardless of which limit is used, if $\itmax$ iterations are
performed, \cref{alg:etcg} returns $s$ as the current CG iterate and
sets $\sout\gets\OKImax$.  This flag indicates that the maximum number
of iterations has been reached while remaining in the interior of the
trust region.

The lemma below motivates the alternative choice for $\itmax$ in
line~\ref{step:cap}.

\begin{lemma} \label{lem:cgits}
  Suppose $\eps I \prec H+2\eps I \preceq (M+2\eps) I$ for $M \in
  [\|H\|,\infty)$ and define
    \begin{equation} \label{eq:cgJe}
      \kappa(M,\eps) := (M+2\eps) / \eps \ \ \text{and} \ \ \Je(M,\eps,\zeta)
      := \tfrac{1}{2}\sqrt{\kappa(M,\eps)}
      \ln\left(4\kappa(M,\eps)^{3/2}/\zeta\right),
    \end{equation}		
    where $\zeta \in (0,1)$ is input to \cref{alg:etcg}.  If
    lines~\ref{step:ifbegin}--\ref{step:ifend} were simply to set
    $\itmax \gets \infty$, then \cref{alg:etcg} would terminate at
    either line~\ref{step:return.outside} or line~\ref{step:return.res}
    after a number of iterations (equivalently, matrix-vector
      products) equal to at most
    \begin{equation} \label{eq:cgits}
      \min\left\{n, \Je(M,\eps,\zeta) \right\} \; = \; 
      \min\left\{n, \tcO(\eps^{-1/2}) \right\}.
    \end{equation}
\end{lemma}
\begin{proof}
  Since $H + 2\eps I \succ \eps I$ by assumption, a direction of
  curvature less than~$\eps$ for $H+2\eps I$ does not exist, meaning
  that termination in line~\ref{step:return.nc} cannot occur.  It
  follows from the fact that $\eps I \prec H + 2\eps I \preceq
  (M+2\eps) I$ and~\cite[proof of Lemma 11]{CWRoyer_SJWright_2018}
  that CG
  would reach an iterate satisfying~\cref{eq:cvcrittcg}---so that
  termination in line~\ref{step:return.res} would occur---in at most
  the number of iterations given by~\cref{eq:cgits}.  Of course, if
  termination occurs earlier in line~\ref{step:return.outside}, the
  bound~\cref{eq:cgits} still holds.
\end{proof}

When employing the trust-region method of Section~\ref{sec:inexact}
for minimizing~$f$, \cref{alg:etcg} is invoked without knowing whether
or not $H + 2\eps I \succ \eps I$.  Nevertheless, \cref{lem:cgits}
allows us to make the following crucial observation.

\begin{lemma}\label{lem:neg-exists}
If the iteration limit in \cref{alg:etcg} is exceeded (that
  is, termination occurs at line~\ref{step:return.itmax}), then $H
\nsucc -\epsilon I$.
\end{lemma}
\begin{proof}
If $\itmax$ is set to $n$ in line~\ref{step:cap} or line~\ref{step:n},
then it would follow from standard CG theory that
\cref{alg:etcg} cannot reach line~\ref{step:return.itmax},
  because either $r_{j+1}=0$ for some $j<n$ (thus termination would
  have occurred at line~\ref{step:res}) or else one of the other
  termination conditions would have been activated before this point.
  Hence, $\itmax$ must have been set in line~\ref{step:cap} to some value less than $n$.  In this case, it follows from
\cref{lem:cgits}, the choice of $M$, and the choice of $\itmax$ in
line~\ref{step:cap} that $H + 2\eps I \nsucc \eps \revised{I}$.
\end{proof}

\floatname{algorithm}{Algorithm}
\begin{algorithm}[t!]
\caption{Truncated CG Method for the Trust-Region Subproblem}
\label{alg:etcg}
\begin{algorithmic}[1]
\STATE \textbf{Input:} Nonzero $g \in\R^n$; $H \in \Sym^n$; regularization parameter $\eps\in(0,\infty)$; trust-region radius $\delta\in(0,\infty)$; accuracy parameter $\zeta \in (0,1)$; flag $\capCG \in \{\true,\false\}$; and (if $\capCG = \true$) upper bound $M \in [\|H\|,\infty)$. 
\STATE \textbf{Output:} trial step $s$ and flag $\sout$ indicating termination type.
\smallskip
\hrule
\smallskip
\IF{$\capCG = \true$}\label{step:ifbegin}
\STATE Set $\itmax \leftarrow \min\left\{n,
\frac12 \sqrt{\kappa}\ln\left(4\kappa^{3/2}/\zeta\right) \right\}$ where $\kappa\leftarrow (M+2\eps)/2$. \label{step:cap}
\ELSE  
    \STATE Set $\itmax \gets n$.\label{step:n} 
\ENDIF\label{step:ifend}
\STATE Set $y_0 \leftarrow 0$, $r_0 \leftarrow g$, $p_0 \leftarrow -g$, 
and $j \leftarrow 0$. 
\WHILE{$j < \itmax$}
\IF{$p_j^T (H + 2\eps I) p_j \le \eps\|p_j\|^2$}\label{step:negcurve_test}
\STATE Compute $\sigma \ge 0$ such that $\|y_j+\sigma\,p_j\|=\delta$.
\STATE \textbf{return} $s \gets y_j+\sigma\,p_j$ and $\sout \gets \OKBneg$. \label{step:return.nc}
\ENDIF
\STATE Set $\alpha_j \leftarrow \|r_j\|^2 / (p_j^T (H + 2\eps I) p_j)$.
\STATE Set $y_{j+1} \leftarrow y_j+\alpha_j p_j$. \label{step:yj}
\IF{$\|y_{j+1}\|\ge \delta$}
\STATE Compute $\sigma \ge 0$ such that $\|y_j+\sigma\,p_j\|=\delta$.
\STATE \textbf{return} $s \gets y_j+\sigma\,p_j$ and $\sout \gets \OKBinf$.  \label{step:return.outside}
\ENDIF
\STATE Set $r_{j+1} \leftarrow r_j+\alpha_j (H + 2\eps I) p_j$.
\IF{$\|r_{j+1}\|\le \tfrac{\zeta}{2} \min\{\|g\|,\eps\|y_{j+1}\|\}$}\label{step:res}
\STATE \textbf{return} $s \gets y_{j+1}$ and $\sout \gets \OKIres$. \label{step:return.res}
\ENDIF
\STATE Set $\beta_{j+1} \leftarrow {(r_{j+1}^T r_{j+1})}/{(r_j^T r_j)}$. \label{step:betaj}
\STATE Set $p_{j+1} \leftarrow  -r_{j+1} + \beta_{j+1}p_j$. \label{step:pj}
\STATE Set $j \gets j + 1$.
\ENDWHILE
\STATE \textbf{return} $s \gets y_{\itmax}$ and $\sout \gets \OKImax$. \label{step:return.itmax}
\end{algorithmic}
\end{algorithm}

\revised{When Algorithm~\ref{alg:etcg} returns because the inequality
  in Line~\ref{step:negcurve_test} holds, it is possible that the
  objective function in~\eqref{eq:inexact1:trsubpb} evaluated at the
  returned vector $s$ is larger than its value at $s=0$, a situation
  that is typically not possible when CG is used as a subproblem
  solver in trust-region methods. This is because, although the
  inequality in Line~\ref{step:negcurve_test} implies that $p_j$ is a
  direction of negative curvature for $H$, $p_j$ is not necessarily a
  direction of negative curvature for the matrix $H+2\epsilon I$ that
  defines the quadratic model in~\eqref{eq:inexact1:trsubpb}. Since,
  in this case, $s$ is obtained by moving to the boundary of the
  trust-region along the direction $p_j$ (similar to the behavior of
  Steihaug's CG method in \cite{TSteihaug_1983}, and needed for our
  complexity result), we require the following result, which
  establishes that any step computed by \cref{alg:etcg} possesses a
  decrease property with respect to the \emph{non-regularized} version
  of the quadratic model.}

\begin{lemma} \label{lem:tcgdecreasestep}
  \revisedbis{The step $s$ returned by \cref{alg:etcg} satisfies}
  \begin{equation*}
    g^T s + \tfrac12 s^T H s \le -\tfrac12\eps\|s\|^2.
  \end{equation*}
\end{lemma}
\begin{proof}
  Basic CG theory ensures that for any $j$ up to termination, the
  sequence $\{g^Ty_j + \tfrac12 y_j^T (H + 2\eps I) y_j\}$ is
  monotonically decreasing.  Since $y_0=0$, we thus have
  \begin{equation}\label{eq:x6W9q}
    g^Ty_i + \tfrac12 y_i^T (H + 2\eps I) y_i \le 0\ \ \text{for all}\ \ i \in \{0,1,\dots,j\}.
  \end{equation}

  \revisedbis{Suppose $\sout \in\{\OKBinf,\OKIres,\OKImax\}$.}
  From \cref{eq:x6W9q} and the fact \revised{(by \cite[Theorem~2.1]{TSteihaug_1983})} that $g^Ts + \tfrac12 s^T(H +
  2\eps I)s \ \revised{\leq}\  g^Ty_j + \tfrac12 y_j^T(H + 2\eps I)y_j$ when $\sout
  \gets \OKBinf$, we have
  \[
    g^T s + \tfrac12 s^T (H + 2\eps I) s \; \le \; 0 \ \Leftrightarrow \ 
    g^T s + \tfrac12 s^T H s \; \le \; -\eps\|s\|^2,	
  \]
  which implies the desired result.
    
  Second, suppose that $\sout \gets \OKBneg$, meaning that
  \cref{alg:etcg} terminates because iteration $j$ yields $p_j^T (H +
  2\eps I) p_j \le \eps \|p_j\|^2$.  If $j=0$, then the fact that
  $p_0=-g$ allows us to conclude that $s = \delta(p_0/\|p_0\|) =
  -\delta (g /\|g\|)$, $\|s\| = \delta$, and
  \[
    \tfrac12 s^T (H + 2\eps I) s = \tfrac12 \delta^2  (p_0^T (H + 2\eps I) p_0) / \|p_0\|^2 \le \tfrac12 \eps \delta^2 = \tfrac12 \eps\|s\|^2,
  \]
  from which it follows that
  \[	
    g^T s + \tfrac12 s^T H s = -\delta\|g\| + \tfrac12 s^T (H+ 2\eps I)s 
    -\eps \|s\|^2 \le -\tfrac12 \eps \|s\|^2, 
  \]
  as desired.  On the other hand, if $j \ge 1$, then the fact that
  $\sout \gets \OKBneg$ means that $s \gets y_j + \sigma p_j$ with
  $\sigma \ge 0$ such that $\|s\|=\delta$. The CG process yields:
  \begin{subequations} \label{eq:cgproperties}
    \begin{align}
      y_i = \sum_{\ell=0}^{i-1} \alpha_\ell p_\ell \in 
      \lspan\left\{p_0,\dotsc,p_{i-1}\right\} 
      &\quad \mbox{for all $i\in\{1,2,\dotsc,j\}$,}\label{eq:cg.yisumpi}  \\
      p_i^T (H + 2\eps I) p_\ell=0  
      &\quad \mbox{for all $\{i,\ell\}\subseteq \{0,1,\dotsc,j\}$ with $i \neq \ell$,}
      \label{eq:cg.piHpj} \\		
      r_i^T p_j = \revised{-\|r_j\|^2} 
      &\quad\mbox{for all $i \in \{0,1,\dotsc,j\}$,} \label{eq:cg.ripj} \\
      \text{and}\ \ y_i^T p_i \ge 0 &\quad \mbox{for all $i\in\{0,1,\dotsc,j\}$.} \label{eq:cg.yiTpi}
    \end{align}
  \end{subequations}
\revised{(Referring to \cref{alg:etcg}, the property
  \cref{eq:cg.yisumpi} follows from Line~\ref{step:yj};
  \cref{eq:cg.piHpj} is the well known conjugacy property; and
  \cref{eq:cg.ripj} is obtained by successively substituting for
  $p_j,p_{j-1}, \dotsc, p_{i+1}$ from Line~\ref{step:pj}, using the
  property that $r_i^Tr_l=0$ for $l \neq i$, and using the definition
  of $\beta_j$ from Line~\ref{step:betaj}.}
  For \cref{eq:cg.yiTpi},
  see \cite[eq.~(2.13)]{TSteihaug_1983}.)
  Together, \cref{eq:cgproperties} and $s = y_j + \sigma p_j$ imply
  \begin{subequations} \label{eq:tcgdecreasestep.Hconj}
    \begin{align}
   g^T p_j &= r_0^T p_j \revised{= -\|r_j \|^2} \le 0
   \\ 
   s^T (H + 2\eps I) s &= y_j^T (H + 2\eps I) y_j
    + \sigma^2 p_j^T (H + 2\eps I) p_j \\ \text{and} \ \
    \|s\|^2 &= \|y_j\|^2 + 2 \sigma y_j^T p_j + \sigma^2 \|p_j\|^2 
    \ge \sigma^2 \|p_j\|^2.
    \end{align}
  \end{subequations}
  Combining \cref{eq:x6W9q}, \cref{eq:tcgdecreasestep.Hconj}, $\sigma \geq 0$, and $p_j^T (H + 2\eps I) p_j \le \eps \|p_j\|^2$ shows that
  \begin{eqnarray*}
    g^T s + \tfrac{1}{2}s^T H s &= &g^T s + \tfrac{1}{2}s^T (H + 2\eps I) s -\eps\|s\|^2 \\
    &= &g^T y_j + \tfrac12 y_j^T (H + 2\eps I) y_j + \sigma g^T p_j + 
    \tfrac12 \sigma^2 p_j^T (H + 2\eps I) p_j -\eps\|s\|^2 \\ 
    &\le &\tfrac12 \sigma^2 p_j^T (H + 2\eps I) p_j -\eps\|s\|^2  
    \le \tfrac12 \sigma^2 \eps \|p_j\|^2 -\eps\|s\|^2 
    \le -\tfrac12 \eps \|s\|^2,
  \end{eqnarray*}
  which completes the proof.
\end{proof}

\revisedbis{\cref{lem:tcgdecreasestep} shows that if $\eps = \epsH$}, 
then the bound on the model decrease obtained by the
truncated CG step $s$ is the same as the bound guaranteed by the
global solution computed for \cref{alg:exact} (see
\cref{lem:exact.modeldec}). However, we note that this decrease is
obtained by using a larger regularization term.

\subsection{A minimum eigenvalue oracle}
\label{subsec:inexact:meo}


The truncated CG algorithm presented in
Section~\ref{subsec:tcg:classic} is only one of the tools we need for
our proposed inexact trust-region Newton method.  Two complicating
cases require an additional tool.

The first case is when $\sout = \OKImax$ is returned by
\cref{alg:etcg}.  In this case, it must hold that the maximum allowed
number of iterations satisfies $\itmax < n$ and, as a consequence of
\cref{lem:neg-exists}, that $H \nsucc -\epsilon I$.  Thus, there
exists a direction of sufficient negative curvature for $H$, and we
need a means of computing one.  The second case is when
\cref{alg:etcg} terminates with $\sout = \OKIres$. In this case, we
only know that the curvature is not sufficiently negative along the
directions computed by the algorithm.  However, it may still be true
that $H \nsucc -\epsilon I$.

These two cases motivate the need for a {\em minimum eigenvalue
  oracle} that estimates the minimum eigenvalue of $H$, or else
returns \revisedbis{an indication} that (with some desired probability) no
sufficiently negative eigenvalue exists. The oracle that we employ is
given by \cref{alg:meo}.



\floatname{algorithm}{Algorithm}
\begin{algorithm}[ht!]
\caption{Minimum Eigenvalue Oracle (MEO)}
\label{alg:meo}
\begin{algorithmic}
  \STATE \textbf{Input:} $g \in \R^n$; $H \in \Sym^n$; regularization parameter $\eps \in (0,\infty)$; trust-region radius $\delta \in(0,\infty)$; failure probability tolerance $\pbpi \in (0,1)$; and $M \in [\|H\|,\infty)$.
  \STATE \textbf{Output:} Either (i) a vector $s = \pm \delta v$ satisfying
  \begin{equation} \label{eq:inexact2.negcurvsk}
  g^T s \le 0, \ \
  s^T H s \le -\tfrac12 \epsilon \|s\|^2, 
  \ \ \text{and} \ \ 
  \|s\|=\delta,
  \end{equation}
  where $v$ has been computed to satisfy $\|v\|=1$ and $v^T H v \leq -
  \epsilon/2$, or (ii) \revisedbis{an indication} that $H \succeq -\epsilon I$
  holds.
  \revisedbis{The probability that the \revisedbis{indication} in case (ii) is made
    yet $H \prec -\epsilon I$ is at most $\pbpi$.}  (The bound $M$ may
  be needed for algorithm termination; see \cref{assum:wccmeo} on
  page~\pageref{assum:wccmeo}.)
\end{algorithmic}
\end{algorithm}

\section{An inexact trust-region Newton method}
\label{sec:inexact}


In this section, we propose a trust-region Newton method that may use,
during each iteration, an \emph{inexact} solution to the trust-region
subproblem computed using the iterative procedures described in
Section~\ref{sec:tcg}. The proposed algorithm is described in
Section~\ref{subsec:inexact:algo} and a second-order complexity
analysis is presented in Section~\ref{subsec:inexact:wcc}.

\subsection{The algorithm}
\label{subsec:inexact:algo}

\cref{alg:inexact2} can be viewed as an inexact version of \cref{alg:exact}.
We aim at remaining close to the traditional Newton-CG approaches
in~\cite{TSteihaug_1983,PhLToint_1981} by having \cref{alg:inexact2}
compute, when appropriate, a truncated CG step in
line~\ref{step:call-cg}. Once such a step is computed (or set to zero
since the current iterate is first-order stationary),
\cref{alg:inexact2} deviates from traditional Newton-CG in
the ``else'' branch (line~\ref{step:deviate.else}), which accounts
for the two situations described in Section~\ref{subsec:inexact:meo},
where an additional check for a negative curvature direction is
needed.  (There is one minor difference: when $\sout = \OKIres$, the MEO
need be called only when $\|g_k\| \leq \epsg$.)

\floatname{algorithm}{Algorithm}

\begin{algorithm}[ht!]
\caption{Trust-Region Newton-CG Method (inexact version)}
\label{alg:inexact2}
\begin{algorithmic}[1]
\REQUIRE Tolerances  $\epsg\in(0,\infty)$ and $\epsH\in(0,\infty)$; \revised{trust-region adjustment} parameters $\gamma_1 \in (0,1)$, $\gamma_2 \in [1,\infty)$, \revised{and $\psi \in (1/\gamma_2,1]$}; 
initial iterate $x_0\in\R^n$; initial trust-region 
radius $\delta_0\in(0,\infty)$; maximum trust-region radius $\delta_{\max} \in [\delta_0,\infty)$; step acceptance parameter $\eta\in(0,1)$; truncated CG accuracy parameter $\zeta \in (0,1)$; MEO failure probability tolerance $\pbpi\in[0,1)$; flag $\capCG\in\{\true,\false\}$; and upper bound $M \in [L_g,\infty)$.
\smallskip
\hrule
\smallskip
\FOR{$k=0,1,2,\dotsc$}
\STATE Evaluate $g_k$ and $H_k$.
\IF{$g_k \neq 0$}
\STATE Call \cref{alg:etcg} with input $g = g_k$, $H = H_k$, $\epsilon = \epsH$, $\delta = \delta_k$, $\zeta$, $\capCG$, and (if $\capCG=\true$) $M$ to compute $\skCG$ and output flag $\sout$. \label{step:call-cg}
\ELSE
\STATE Set $\skCG \gets 0$ and $\sout \gets \OKIres$.
\ENDIF
\IF{$\sout\in\{\OKBneg,\OKBinf\}$ or ($\| g_k \| > \epsg$ and $\sout = \OKIres$)} \label{step:deviate}
\STATE Set $s_k \gets \skCG$.
\ELSE[that is, $\sout = \OKImax$ or ($\|g_k\| \leq \epsg$ and $\sout = \OKIres$)] \label{step:deviate.else}
\STATE Call \cref{alg:meo} with inputs $g = g_k$, $H = H_k$, $\eps = \epsH$, $\delta = \delta_k$, $\pbpi$, and $M$, obtaining either $s_k$ satisfying \cref{eq:inexact2.negcurvsk} or \revisedbis{an indication} that $H_k \succeq -\epsH I$. \label{step:call-meo}
\IF{\cref{alg:meo} predicts that $H_k \succeq -\epsH I$}
  \STATE \textbf{return} $x_k$. \label{step:return-inexact}
\ENDIF
\ENDIF
\STATE Compute the ratio of actual to predicted decrease in $f$ defined as
\[
\rho_k \gets \frac{f(x_k)-f(x_k+s_k)}{m_k(x_k)-m_k(x_k+s_k)}.
\]
\IF{$\rho_k \ge \eta$}
\STATE Set $x_{k+1} \gets x_k+s_k$.
\revised{\IF{$\|s_k\| \ge \psi \delta_k$}
\STATE Set $\delta_{k+1} \gets \min\left\{\gamma_2\delta_k,\delta_{\max}\right\}$.
\ELSE \STATE Set $\delta_{k+1} \gets \delta_k$.
\ENDIF}
\ELSE
\STATE Set $x_{k+1} \gets x_k$ and $\delta_{k+1} \gets \gamma_1 \revised{ \|s_k\|}$.
\ENDIF
\ENDFOR
\end{algorithmic}
\end{algorithm}

\subsection{Complexity} 
\label{subsec:inexact:wcc}


As in~\cite{CWRoyer_MONeill_SJWright_2019}, we make the following
assumption on the MEO in order to obtain complexity results for
\cref{alg:inexact2}.

\begin{assumption} \label{assum:wccmeo}
When \cref{alg:meo} is called by \cref{alg:inexact2}, the number of
Hessian-vector products required is no more than
\begin{equation} \label{eq:wccmeo}
  N_{\mathrm{meo}}
  = N_{\mathrm{meo}}(\epsH) 
  :=\min\left\{n,1+ 
  \left\lceil\Cmeo\epsH^{-1/2} \right\rceil\right\}
\end{equation}
where the quantity $\Cmeo$ depends at most logarithmically on
$\pbpi$.
\end{assumption}

The following instances of \cref{alg:meo} satisfy \cref{assum:wccmeo}.
\begin{itemize}
\item The \emph{Lanczos algorithm} applied to $H$ starting with a
  random vector uniformly distributed on the unit sphere.  For any
  $\pbpi \in (0,1)$, this satisfies the conditions in
    \cref{assum:wccmeo} with $\Cmeo = \ln(2.75
    n/\pbpi^2)\sqrt{M}/2$;
    see~\cite[Lemma~2]{CWRoyer_MONeill_SJWright_2019}.
  \item The \emph{conjugate gradient algorithm} applied to $\left(H +
    \tfrac{\epsH}{2} I \right) s = b$, where $b$ is a random vector
    uniformly distributed on the unit sphere.  For any $\pbpi \in
    (0,1)$, this offers \cref{assum:wccmeo} with the same value of
      $\Cmeo$ as in the Lanczos-based approach;
      see~\cite[Theorem~1]{CWRoyer_MONeill_SJWright_2019}.
\end{itemize}
  Since for each instance the conditions of \cref{assum:wccmeo} hold
  with $\Cmeo$ equal to the given value, it follows that throughout a
  run of \cref{alg:inexact2}, the conditions hold with $\Cmeo =
  \ln(2.75 n/\pbpi^2)\sqrt{L_g}/2$.  \cref{alg:meo} could also be
  implemented by means of an exact (minimum) eigenvalue calculation of
  the Hessian. In that case, up to $n$ Hessian-vector products may be
  required to evaluate the full Hessian.



In the following analysis, we use similar notation as in
Section~\ref{sec:exact}, although the analysis here is notably
different due to the randomness of the MEO.  For consistency, we use
the same definitions of the index sets $\Kcal$, $\Kcalis$, $\Kcalbs$,
$\Scal$, and $\Ucal$ that appear in the beginning of
Section~\ref{subsec:exact:wcc}; in particular, we define $\Kcal$ as
the index set of iterations completed prior to termination.  However,
note that for \cref{alg:inexact2} these sets are random variables, in
the sense that for the same objective function and algorithm inputs,
they may have different realizations due to the randomness
in~\cref{alg:meo}.  Thus, when we refer, for example, to $k \in
\Kcal$, we are referring to $k \in \Kcal$ \emph{for a given
  realization of a run of \cref{alg:inexact2}.}  We also prove bounds
on quantities that are shown to hold for \emph{all} realizations of a
run of the algorithm (for a given objective function and algorithm
inputs).  To emphasize that these bounds hold for all realizations,
their constants are written with a bar over the letter in the
definition.

Our first result provides a lower bound on the reduction in the
quadratic model of the objective function achieved by each trial step.

\begin{lemma} \label{lem:inexact2.modeldec}
  Consider any realization of a run of \cref{alg:inexact2}.  For all
  $k \in \Kcal$,
  \[
    m_k(x_k) - m_k(x_k+s_k) \ge \tfrac{1}{4}\epsH \|s_k\|^2.
  \]
\end{lemma}
\begin{proof}
  If $s_k = \skCG$, where $\skCG$ is computed from \cref{alg:etcg},
  then, \revisedbis{it follows by \cref{lem:tcgdecreasestep} that 
  the desired bound holds.}  Now suppose that~$s_k$ is computed
  from \cref{alg:meo} in line~\ref{step:call-meo}. Since $k \in
  \Kcal$, \cref{alg:inexact2} does not terminate in iteration $k$, and
  it follows from \cref{eq:inexact2.negcurvsk} that
  \[
    m_k(x_k) - m_k(x_k+s_k) 
    = -g_k^T s_k - \tfrac{1}{2} s_k^T H_k s_k 
    \ge -\tfrac{1}{2}s_k^T H_k s_k \ge \tfrac{1}{4} \epsH \|s_k\|^2,
  \]
  as desired.
\end{proof}

We can now show that a sufficiently small trust-region radius leads to
a successful iteration, and provide a lower bound on the sequence of
trust-region radii.

\begin{lemma} \label{lem:inexact2.delta-lb}
  Consider any realization of a run of \cref{alg:inexact2}.  For all
  $k \in \Kcal$,
  \revised{if $k \in \Ucal$, then $\delta_k > 3(1-\eta) \epsH/(2\lipH)$.} Hence, by the trust-region radius update procedure, it
  follows that for any realization of a run of \cref{alg:inexact2} that
  \begin{equation} \label{eq:inexact2.delta-lb}
    \delta_k \geq \bar\delta_{\min} := \min\left\{\delta_0, 
     \left(\tfrac{3\gamma_1(1-\eta)}{2\lipH}\right) \epsH \right\} \in (0,\infty) 
    \ \ \text{for all}\ \ k \in \Kcal.
  \end{equation}
\end{lemma}
\begin{proof}
  For any realization of a run of the algorithm, we can follow the 
  proof of \cref{lem:exact:delta-lb}, using 
  \cref{lem:inexact2.modeldec} in lieu of \cref{lem:exact.modeldec}. 
  Hence, the lower bound in \cref{eq:inexact2.delta-lb} holds,
  where~$\bar\delta_{\min}$ is independent of any
  particular realization of a run.
%
\end{proof}

\medskip

We now establish a bound on the objective reduction for a successful step.

\begin{lemma} \label{lem:inexact2.actualdec}
  Consider any realization of a run of \cref{alg:inexact2}.  The following hold for all successful iterations:
  \begin{enumerate}[(i)]
  \item If $k \in \Kcalbs \cap \Scal$, then
    \[	
    f_k - f_{k+1} \ge \tfrac{\eta}4 \epsH \delta_k^2.
    \]
  \item If $k \in \Kcalis \cap \Scal$,
    then $\|g_k\| > \epsg$, $\sout = \OKIres$, and 
    \[
    f_k - f_{k+1} \ge \tfrac{\eta}{4(7+2\lipH)} \min\left\{
    \|g_{k+1}\|^2 \epsH^{-1}, \epsH^3 \right\}.
    \]
  \end{enumerate}
\end{lemma}
\begin{proof}
  For part (i), we combine $k\in\Kcalbs\cap\Scal$ with \cref{lem:inexact2.modeldec} to obtain, as desired,
  \[
  f_k - f_{k+1} 
  \ge \eta\left(m_k(x_k) - m_k(x_k+s_k) \right) 
  \ge \tfrac{\eta}4 \epsH \|s_k\|^2 
  =  \tfrac{\eta}4 \epsH \delta_k^2.
  \]
  
  Now consider part (ii). Note that since $k \in \Kcalis$, $s_k$
  cannot have been computed from a call to \cref{alg:meo} in
  line~\ref{step:call-meo}, since such steps always have $\|s_k \|=
  \delta_k$. Thus, $s_k=\skCG$.
  Moreover, from line~\ref{step:deviate} and the fact that $k \in
  \Kcalis$, we have that $\|g_k\| > \epsg$ and $\sout = \OKIres$, as
  desired.  In turn, the fact that $\sout = \OKIres$ implies
  that~\cref{eq:cvcrittcg} holds with $H = H_k$, $g = g_k$, $s = s_k$,
  and $\epsilon = \epsH$ so that
  \begin{equation} \label{eq:inexact.internalstep}
    r_k := (H_k + 2\epsH) s_k + g_k
    \ \ \text{has} \ \
    \|r_k\| \le \tfrac{\zeta}{2}\epsH\|s_k\|.
  \end{equation}
  Combining this bound with \cref{eq:fub.3} and $\zeta\in(0,1)$, we have
  \begin{align*}
    \|g_{k+1}\| &= \|g_{k+1} - g_k - (H_k+2\epsH)s_k + r_k\| \\
    &\le \|g_{k+1} - g_k - H_k s_k\| + 2\epsH\|s_k\| + \|r_k\| \\
    &\le \tfrac{\lipH}{2}\|s_k\|^2 + \left(\tfrac{4+\zeta}{2}\right)\epsH \|s_k\| 
    \le \tfrac{\lipH}{2}\|s_k\|^2 + \tfrac{5}{2}\epsH \|s_k\|,
  \end{align*}
  which can be rearranged to yield
  \[
  \tfrac{\lipH}{2}\|s_k\|^2+ \tfrac{5}{2}\epsH \|s_k\| - \|g_{k+1}\| \ge 0.
  \]  
  Reasoning as in the proof of \cref{lem:exact.actualdec},
  with $\frac{5}{2}\epsH$ replacing $\epsH$, we obtain
  \begin{eqnarray*}
    \|s_k\| 
    &\ge &\tfrac{-\tfrac{5}{2}\epsH + 
      \sqrt{\left(\tfrac{5}{2}\right)^2\epsH^2 + 2 \lipH \|g_{k+1}\|}}{\lipH}
    = \left(\tfrac{-5+\sqrt{25 + 8 \lipH  \|g_{k+1}\|\epsH^{-2}}}{2\lipH} \right) \epsH.
  \end{eqnarray*}
  \revised{
  By setting $(a,b,t) = (5, 8\lipH, \|g_{k+1}\|\epsH^{-2})$ in the
  inequality \eqref{eq:ws8}, we have}
  \begin{align*}
    \|s_k\|
    &\geq \left(\tfrac{-5+\sqrt{25+8\lipH}}{2\lipH}\right)
    \min\left\{ \|g_{k+1}\|\epsH^{-2},1\right\}\epsH \\
    &= \left(\tfrac{8 \lipH}{2\lipH (5+\sqrt{25+8\lipH})}\right)
    \min\left\{ \|g_{k+1}\|\epsH^{-1},\epsH\right\}  \\
    &= \left(\tfrac{4}{5+\sqrt{25+8\lipH}}\right)
    \min\left\{ \|g_{k+1}\|\epsH^{-1},\epsH\right\} \\
    &\geq \left(\tfrac{2}{\sqrt{25+8\lipH}}\right)
    \min\left\{ \|g_{k+1}\|\epsH^{-1},\epsH\right\}
    \geq  \tfrac{1}{\sqrt{7+2\lipH}}
    \min\left\{ \|g_{k+1}\|\epsH^{-1},\epsH\right\},
  \end{align*}
  which may be combined with $k\in\Kcalis\cap\Scal$ and \cref{lem:inexact2.modeldec} to obtain
  \begin{align*}
    f_k - f_{k+1} 
    &\ge 
    \eta(m_k(x_k)-m_k(x_{k+1})) 
    \ge \tfrac14  {\eta}\epsH \|s_k\|^2  
    \ge \tfrac{\eta}{4(7+2\lipH)}
    \min\left\{ \|g_{k+1}\|^2 \epsH^{-1},\epsH^3\right\},
  \end{align*}
  which completes the proof. 
\end{proof}


The next result is analogous to \cref{lem:exact.wccsuccc} and takes
randomness in the MEO into account.

\begin{lemma} \label{lem:inexact2.wccsucc}
  For any realization of a run of \cref{alg:inexact2}, the number of successful iterations performed before termination occurs satisfies 
    \begin{equation} \label{eq:inexact2.wccitssuccValue}
      |\Scal| 
      \leq \bar{K}_{\Scal}(\epsg,\epsH)
      := \left\lfloor\mathcal{\bar C}_\cS
      \max\{\epsH^{-1}, \epsg^{-2}\epsH, \epsH^{-3}\} \right\rfloor + 1,
    \end{equation}
    where
    \begin{equation} \label{eq:inexact2.wccitsCc}
      \mathcal{\bar C}_{\cS} 
      := \tfrac{8(f_0 - \flow)}{\eta}
      \max\left\{\tfrac{1}{\delta_0^2},
      \tfrac{4\lipH^2}{9\gamma_1^2(1-\eta)^2},
      7+2\lipH)\right\}.
    \end{equation}
\end{lemma}
\begin{proof}
    For a given realization of a run of the algorithm, we can 
    follow the reasoning of the proof for \cref{lem:exact.wccsuccc}.
    In what follows, $\cSleq$, $\cSgg$, and $\cSgleq$ are defined as in the 
    proof of Lemma~\ref{lem:exact.wccsuccc}.  (As is the case for $\Kcal$, we note 
    that these index sets are now realizations of random index sets.)

Consider first $k \in \cSleq$, and let us define the constant
  \begin{equation} \label{def:c1}
    c_1 := \tfrac{\eta}{4} \min\left\{
  		\delta_0^2,
		\tfrac{9\gamma_1^2(1-\eta)^2}{4 \lipH^2}, 
  		\tfrac{1}{7+2\lipH}\right\}.
  \end{equation}
  We can use
\cref{lem:inexact2.actualdec} to conclude that $k\in\Kcal\cap\Kcalbs$,
so that $\|s_k\| = \delta_k$. By combining
\cref{lem:inexact2.actualdec}(i), \cref{lem:inexact2.delta-lb}, and \cref{def:c1}, we
have for $k \in  \cSleq$ that
  \begin{equation} \label{eq:xj6}
    f_k-f_{k+1} \ge \tfrac{\eta}{4} \epsH \delta_k^2 \ge
    \tfrac{\eta}{4} \min \left\{ \delta_0^2 \epsH, \tfrac{9 \gamma_1^2(1-\eta)^2}{4 \lipH^2} \epsH^3\right\} \ge c_1 \min \left\{ \epsH,\epsH^{\revised{3}} \right\}.
  \end{equation}

For $k \in \cSgg$, we have from \cref{lem:inexact2.actualdec}
  (either (i) or (ii)) and \cref{lem:inexact2.delta-lb} that
  \begin{align}
    \nonumber
    f_k-f_{k+1} & \ge \tfrac{\eta}{4} \min \left\{ \delta_0^2 \epsH, \tfrac{9 \gamma_1^2(1-\eta)^2}{4\lipH^2} \epsH^3, \tfrac{1}{7+2\lipH} \epsg^2 \epsH^{-1}, \tfrac{1}{7+2\lipH} \epsH^3\right\} \\
    \label{eq:xj7}
    & \ge c_1 \min \left\{ \epsH, \epsH^3,\epsg^2 \epsH^{-1} \right\}.
    \end{align}
  By following the reasoning that led to \cref{bounds-first2}, we
  obtain from \cref{eq:xj6} and \cref{eq:xj7} that
  \[
  |\cSleq| +   |\cSgg| \leq \left(\tfrac{f_0 - \flow}{c_1}\right)\max\{\epsH^{-1},\epsg^{-2}\epsH,\epsH^{-3}\}.
  \]
  As in the proof of \cref{lem:exact.wccsuccc}, we have that $ |\cSgleq| \leq
  |\cSleq| + 1 $, so that
  \[
  |\cS|
    = |\cSleq| + |\cSgg| + |\cSgleq|
    \leq \tfrac{2(f_0-\flow)}{c_1}\max\{\epsH^{-1},\epsg^{-2}\epsH,\epsH^{-3}\} + 1.
    \]
    The desired bound follows by substituting the definition
    \cref{def:c1} into this bound.
		To complete the proof, we note
    that the right-hand side of~\cref{eq:inexact2.wccitssuccValue} is 
    identical for any realization of the algorithm run with the same inputs.
\end{proof}

We now provide a bound on the maximum number of unsuccessful iterations.


%
%

\begin{lemma} \label{lem:inexact2.wccunsuccsucc}
For any realization of a run of \cref{alg:inexact2}, the number of
unsuccessful iterations performed before termination occurs
\revised{is either zero or else satisfies}
  \begin{equation} \label{eq:inexact2.wccunsuccsucc}
    |\Ucal| \leq 
    \left\lfloor 1+
    \log_{\gamma_1}\left(\tfrac{3(1-\eta)}{2\lipH\delta_{\max}}\right)+ 
    \log_{\gamma_1}\left(\epsH\right)
    \right\rfloor (|\Scal|+1).
  \end{equation}
\end{lemma}
\begin{proof}
For a given realization of a run of the algorithm, the bound follows
from the argument in the proof of \cref{lem:exact.wccunsuccsucc} with
two changes. First, \cref{lem:inexact2.delta-lb} is used in place of
\cref{lem:exact:delta-lb}. Second, we do not know that the iteration
immediately prior to termination must be a successful iteration for
\cref{alg:inexact2}, as was the case for \cref{alg:exact}.  However,
using the argument in the proof of \cref{lem:exact.wccunsuccsucc}
along with \cref{lem:inexact2.delta-lb} shows that \revised{if a
  sequence of consecutive unsuccessful iterations is taken after the
  final successful iteration, there can be no more than}
\[
\left\lfloor 1 + 
  \log_{\gamma_1}\left(\tfrac{3(1-\eta)}{2\lipH \delta_{\max}}\right) +
  \log_{\gamma_1}(\epsH) \right\rfloor
  \]
\revised{such iterations in this sequence, since} otherwise an
additional successful iteration would be performed.  Taking this fact
into account leads to the extra $1$ on the right-hand side of
\cref{eq:inexact2.wccunsuccsucc} as compared
to~\cref{eq:exact.wccunsuccsucc}.
\end{proof}



\revisedbis{We assume in our remaining complexity results that the
  following common-sense rule is used in the implementation of
  \cref{alg:inexact2}.
\begin{Str}\label{impstr}
For any realization of a run of
\cref{alg:inexact2}, suppose $k \in \Kcal$ is an index of an iteration
such that (i) \cref{alg:meo} is called in line~\ref{step:call-meo} and
returns a negative curvature direction $s_k$ for~$H_k$ and (ii) the
step $s_k$ is subsequently rejected (that is, $k \in
\Ucal$).  Then, the negative curvature direction (call it $v = v_k$)
used to compute $s_k$ is stored and used until the next successful
iteration. Until then, every call to \cref{alg:meo} is replaced by an
access to $v_k$, scaled appropriately to compute $s_k$ with norm
$\delta_k$.
\end{Str}
}

\medskip


\revisedbis{This strategy} implies that \cref{alg:inexact2}
cannot terminate following a sequence of unsuccessful iterations if
any one of them yields a direction of sufficiently negative curvature.
In practice, this means that \cref{alg:inexact2} calls \cref{alg:meo}
at most once between successful iterations.
In the next iteration complexity result, this assumption is used to
obtain the probabilistic result for returning an
$(\epsg,\epsH)$-stationarity point.

\begin{theorem} \label{theo:inexact2.wccorder}
  Under \cref{assum:f}, for any realization of a run, the number of
  successful iterations $($and objective gradient evaluations$)$
  performed by \cref{alg:inexact2} before termination occurs satisfies
  $($with $\bar{K}_{\Scal}(\epsg,\epsH)$ defined in
  \cref{eq:inexact2.wccitssuccValue}$)$
  \begin{equation} \label{eq:inexact2.wccordersuccits}
    |\Scal| \leq \bar{K}_{\Scal}(\epsg,\epsH) = \mathcal{O} \left( \max\{\epsH^{-3},\epsH^{-1},\epsg^{-2}\epsH\}\right)
  \end{equation}
  and the total number of iterations (and objective function
  evaluations) performed before termination occurs satisfies
  \begin{equation} \label{eq:inexact2.wccorderits}
    \begin{aligned}
      |\Kcal| 
      &\leq \left\lfloor 1 + 
      \log_{\gamma_1}\left(\tfrac{3(1-\eta)}{2\lipH\delta_{\max}}\right)+
      \log_{\gamma_1}(\epsH)\right\rfloor(\bar{K}_{\Scal}(\epsg,\epsH)+1) \\
      &= \mathcal{O}\left(
      \log_{1/\gamma_1}(\epsH^{-1}) \max\{\epsH^{-3},\epsH^{-1}, \epsg^{-2}\epsH\}
      \right).
    \end{aligned}
  \end{equation}
  If $\capCG = \false$, then $\|g_k\| \leq \epsg$ holds at
  termination.  In any case,
 \revisedbis{given Implementation Strategy~4.1}, the vector $x_k$ returned by \cref{alg:inexact2} is an
  $(\epsg,\epsH)$-stationary point with probability at least
  $(1-\pbpi)^{\bar{K}_{\Scal}(\epsg,\epsH)}$.
\end{theorem}
\begin{proof}
The results in \cref{eq:inexact2.wccordersuccits} and
\cref{eq:inexact2.wccorderits} follow from \cref{lem:inexact2.wccsucc}
and \cref{lem:inexact2.wccunsuccsucc}.  If $\capCG = \false$, then the
flag output by \cref{alg:etcg} has $\sout \neq \OKImax$.  Combining
this fact with line~\ref{step:deviate} of \cref{alg:inexact2} allows
us to conclude that $\|g_k\| \leq \epsg$ when termination occurs in
this case.

\revisedbis{Suppose that Implementation Strategy~4.1 is used. Then,}
the vector $x_k$ returned by \cref{alg:inexact2} is not an
$(\epsg,\epsH)$-stationary point only if the MEO (\cref{alg:meo})
makes an inaccurate \revisedbis{indication of
  near-positive-definiteness}, which, each time it is called, can
occur with probability at most $\pbpi$.  \revisedbis{Our goal now is to prove that the vector~$x_k$ returned by \cref{alg:inexact2} is an
  $(\epsg,\epsH)$-stationary point with probability at least
  $(1-\pbpi)^{\bar{K}_{\Scal}(\epsg,\epsH)}$.  To that end, for all $k \in \N{}$, let ${\tilde P}_k$ be the probability that the algorithm reaches iteration $k$ and $x_k$ is not $(\epsg,\epsH)$-stationary.  Similarly, for all $k \in \N{}$, let $P_k$ be the probability that the algorithm reaches iteration $k$ and $x_k$ is not $(\epsg,\epsH)$-stationary, yet the algorithm terminates in iteration $k$ due to an inaccurate indication from the MEO.  It follows from the aforementioned property of the MEO in this setting that $P_k \leq \xi {\tilde P}_k$ for all $k \in \N{}$.  Thus, since it is trivially true that
  \begin{equation*}
    {\tilde P}_k + \sum_{i=0}^{k-1} P_i \leq 1\ \ \text{for all}\ \ k \in \N{},
  \end{equation*}
  it follows that
  \begin{equation}\label{eq.new_prob}
    P_k \leq \pbpi {\tilde P}_k \leq \pbpi \left(1-\sum_{i=0}^{k-1} P_i \right)\ \ \text{for all}\ \ k \in \N{}.
  \end{equation}
  We now define $M_k$ to be the number of calls to the MEO that have
  occurred up to and including iteration~$k$, for any $k \in \N{}$.
  Let us now prove by induction that $\sum_{i=0}^k P_i \le 1-
  (1-\pbpi)^{M_k}$ for all $k \in \N{}$.  For $k=0$, the claim holds
  trivially both when $M_0=0$ (in which case $P_0 = 0$) and when
  $M_0=1$ (in which case $P_0 \le \pbpi$).  Now suppose that the claim
  is true for some $k \in \N{}$; we aim to prove that it remains true
  for $k+1$.  If the algorithm reaches iteration $k+1$ in which
  $x_{k+1}$ is not $(\epsg,\epsH)$-stationary and the MEO is {\em not}
  called in iteration $k+1$, then $M_{k+1} = M_k$ and $P_{k+1} = 0$,
  so by the induction hypothesis it follows that
  \begin{equation*}
    \sum_{i=0}^{k+1} P_i = \sum_{i=0}^k P_i \le 1-
  (1-\pbpi)^{M_k} = 1-
  (1-\pbpi)^{M_{k+1}},
  \end{equation*}
  as desired.   On the other hand, if the algorithm reaches iteration $k+1$, the iterate $x_{k+1}$ is not $(\epsg,\epsH)$-stationary, and the MEO {\em is} called in
  iteration $k+1$, then $M_{k+1}=M_k+1$ and along with \cref{eq.new_prob} and the induction hypothesis it follows that
  \begin{align*}
    \sum_{i=0}^{k+1} P_i & = \sum_{i=0}^k P_i  + P_{k+1} \\
    &\le  \sum_{i=0}^k P_i  + \pbpi \left(1-\sum_{i=0}^k P_i \right) \\
    &= \pbpi + (1-\pbpi) \sum_{i=0}^k P_i \\
    & \le \pbpi + (1-\pbpi) (1-(1-\pbpi)^{M_k}) \\
    &= 1-(1-\pbpi)^{M_k+1} = 1-(1-\pbpi)^{M_{k+1}},
    \end{align*}
as desired, again. Since $| \Scal|$ is an upper bound on the number of
successful steps, which in turn bounds the number of calls to MEO when
Implementation Strategy~4.1 is used, we have $M_k \le |\Scal| \le
\bar{K}_{\Scal}(\epsg,\epsH)$ for all $k$.  We conclude that the
probability that the algorithm terminates incorrectly due to an
incorrect indication from MEO on {\em any} iteration is bounded above
by $1-(1-\pbpi)^{\bar{K}_{\Scal}(\epsg,\epsH)}$.  Consequently, the probability 
that the algorithm outputs an $(\epsg,\epsH)$-stationary point 
is at least $(1-\pbpi)^{\bar{K}_{\Scal}(\epsg,\epsH)}$, as claimed.}
\end{proof}


Finally, we state a complexity result for the number of Hessian-vector
products. For simplicity, we focus on the case of a small tolerance 
$\epsg$.

\begin{theorem}\label{thm:inexact2.wccHvprod}
  Let \cref{assum:f} and \cref{assum:wccmeo} hold, \revisedbis{and suppose that Implementation Strategy~4.1 is used.}
  Suppose that $\epsH = \epsg^{1/2}$ with $\epsg\in(0,1)$.  Then, for
  any realization of a run of \cref{alg:inexact2}, the total number of
  Hessian-vector products performed satisfies:
  \begin{itemize}
  \item[(i)] If $\capCG = \false$, then the number of Hessian-vector products is bounded by
  \[
  n|\Kcal| + N_{\mathrm{meo}}(\epsH)|\bar K_{\Scal}(\epsg,\epsH)| 
  = n \tcO (\epsg^{-3/2}).
  \]
  \item[(ii)] If $\capCG = \true$, then the number of Hessian-vector
    products is bounded by
    \[
    \min\{n,J(L_g,\epsH,\zeta)\}|\Kcal| +
    N_{\mathrm{meo}}(\epsH) |\bar K_{\Scal}(\epsg,\epsH)| 
    =
    \min\{n,\epsg^{-1/4}\} \tcO(\epsg^{-3/2}).
   \]
  \end{itemize}
\end{theorem}
\begin{proof}
First, suppose $\capCG = \false$.  Then, for any $k \in \Kcal$ in any
realization of a run of the algorithm, the maximum number of
Hessian-vector products computed by the truncated CG algorithm is $n$.
In addition, over any realization of a run, the maximum number of
Hessian-vector products computed by the MEO is
$N_{\mathrm{meo}}(\epsH)$ (see \cref{eq:wccmeo}) each of the (at most)
$|\bar K_{\Scal}(\epsg,\epsH)|$ times it is called.  Since the number
of Hessian-vector products performed by \cref{alg:inexact2} is the sum
of these two, we have proved the left-hand side of part (i). For the
estimate $n \tcO(\epsg^{-3/2})$, we use $\epsH=\epsg^{1/2}$, the bound
on $\Kcal$ from \cref{theo:inexact2.wccorder}, the estimate of
$N_{\mathrm{meo}}$ from \cref{assum:wccmeo}, and
the fact that $\max \{ \epsH^{-3},\epsH^{-1},\epsg^{-2} \epsH \} =
\max \{ \epsg^{-3/2},\epsg^{-1/2} \} = \epsg^{-3/2}$ when
$\epsg\in(0,1)$.

For part (ii), we use the same estimates as well as the estimate of
$J(L_g,\epsH,\zeta)$ from \cref{lem:cgits} and \cref{eq:fub.4}, noting
that both $|\Kcal|$ and $\bar{K}_{\Scal}(\epsg,\epsH)$ are
$\tcO(\epsg^{-3/2})$ while $J(L_g,\epsH,\zeta)$ and $N_{\mathrm{meo}}$
are both $\min\{n,\tcO(\epsg^{-1/4})\}$.
\end{proof}

Note that for $n \gg \epsg^{-1/4}$, the bound in part (ii) of
  this theorem is $\tcO(\epsg^{-7/4})$, which is a familiar quantity
  in the literature on the operation complexity required to
  find an $(\epsg,\epsg^{1/2})$-stationary point~\cite{NAgarwal_ZAllenZhu_BBullins_EHazan_TMa_2017,YCarmon_JCDuchi_OHinder_ASidford_2018a,CWRoyer_SJWright_2018}.

\cref{thm:inexact2.wccHvprod} illustrates the benefits of using a
capped truncated CG routine in terms of attaining good computational
complexity guarantees. As a final remark, we expect the ``cap'' of
Algorithm~\ref{alg:etcg} to be triggered only in rare cases, due to
the conservative nature of the CG convergence bounds that gave rise to
this cap.

\section{Practical considerations}
\label{sec:practice}

\revised{ Having presented an analysis of the theoretical complexity of
  our inexact trust-region Newton-CG approach, we consider several
  practical issues that arise in developing a computational
  implementation of this method.}

\revised{ The randomness inherent in \cref{alg:meo} is central to the
  complexity analysis of \cref{alg:inexact2} that was presented
  in Section~\ref{sec:inexact}. Since the randomness carries with it a small probability of
  failure of \cref{alg:meo}, two unsavory situations can occur that
  lead to ``failure modes'' for \cref{alg:inexact2}.}  First, suppose
that \cref{alg:meo} is called in line~\ref{step:call-meo} because
$\|g_k\| \leq \epsg$ and $\sout = \OKIres$. In this case, if $\capCG =
\true$ and \cref{alg:meo} predicts that $\lambdamin(H_k) \geq -\epsH$,
then with probability up to $\pbpi$ this \revisedbis{indication} is incorrect and a
direction of sufficiently negative curvature actually exists but was
not found.  Second, suppose that \cref{alg:meo} is called because
$\sout = \OKImax$.  Here, it is again possible with probability up to
$\pbpi$ that the \revisedbis{indication} $\lambdamin(H_k) > -\epsH$ will be made,
even though we know from \cref{lem:neg-exists} that $\lambdamin(H_k)
\leq -\epsH$.  This second case can occur even when $\|g_k\| > \epsg$,
meaning that termination can occur at a point that is not even
$\epsg$-stationary. Note, however, that in a given iteration the
probability of these two situations is bounded by $\pbpi$, which
appears only logarithmically in the constant $\Cmeo$ of
\cref{assum:wccmeo}, and thus can be chosen to be extremely
small. \revised{We can avoid this second case by replacing
  \cref{alg:etcg} with the alternative truncated CG method of
  \cite[Algorithm~1]{CWRoyer_MONeill_SJWright_2019}, which computes
  and returns a negative curvature direction whenever it detects that
  one exists. This ``implicitly capped'' method is more complicated to
  describe than the ``explicitly capped'' version of CG that we
  consider here, so in this paper we opted for simplicity of
  description at the expense of a (small) probability of failure in
  the subsequent call to \cref{alg:meo}.}

\revised{Our experiments with the inexact algorithms reported in
  Section~\ref{sec:numer} use the randomized Lanczos procedure of
  \cite{JKuczynski_HWozniakowski_1992,JKuczynski_HWozniakowski_1994}.
  In practice, \cref{alg:meo} is rarely invoked by
  \cref{alg:inexact2}---in the vast majority of test problems, it is
  invoked only once on the last iteration of \cref{alg:inexact2} as a
  final check that the Hessian is approximately positive semidefinite.
  Rather than explicitly capping the number of iterations to the value
  described in \cref{assum:wccmeo} and the comments that follow it, we
  use an adaptive criterion to decide when a close approximation to
  the minimum eigenvalue has been detected. Specifically, we stop at
  iteration $l$ if $\lambda_{l-t} - \lambda_l \le
  10^{-5}$, 
    where $t = \min\{l,n,10\}$ and
  $\lambda_l$ is the estimate of the minimum eigenvalue at the $l$th
  iteration of randomized Lanczos. It makes sense to use such a tight
  tolerance, since \cref{alg:meo} is rarely invoked by
  \cref{alg:inexact2}, so that the cost of doing a careful check for
  approximate semidefiniteness is worthwhile.}

\revised{In our implementations of the Truncated CG procedure
  (\cref{alg:etcg}), we use the quantity $\bar{n} := \min\{n+2,1.2n\}$
  in place of $n$, in determining the maximum number of iterations
  $k_{\max}$. This relaxation can be beneficial in practice since loss of conjugacy
  due to numerical rounding can result in a zero residual {\em not}
  being attained by CG after $n$ steps. Typically, a small number of
  additional iterations beyond $n$ suffices to obtain a more accurate
  solution. Another key parameter in this algorithm, the value $\zeta$
  used as a convergence threshold for the residual, is set to $.25$ in
  our tests. (We discuss the settings of the parameters in
  \cref{alg:inexact2} in the next section.) }

\section{Computational experiments} 
\label{sec:numer}

We implemented several variants of trust-region Newton methods in
Matlab, as follows.
\begin{itemize}
  \item \emph{TR-Newton}.  An implementation of \cref{alg:exact} with
    the trust-region subproblem solved using a Mor\'e-Sorensen
    approach \cite{JJMore_DCSorensen_1983}.
  \item \emph{TR-Newton (no reg.)}.  The same as \emph{TR-Newton},
    \revisedbis{except that the regularization term involving $\epsH$ is 
    omitted from the subproblem objective}
    \cref{eq:algos:exacttrsubpb}.  This variant demonstrates the
    effect of this regularization term on the practical performance of
    \emph{TR-Newton}.
  \item \emph{TR-Newton-CG-explicit}.  An implementation of
    \cref{alg:inexact2} with an explicit cap on the number of CG
    iterations (that is, $\capCG = \true$).
  \item \revised{\emph{TR-Newton-CG-explicit (no reg.)}. The same as 
  \emph{TR-Newton-CG-explicit}, except that the regularization 
  term involving $\epsH$ is omitted from the subproblem objective, 
  again to illustrate the impact of this regularization.}
  \item \revised{\emph{TR-Newton-CG}.  An implementation of 
    \cref{alg:inexact2} without an explicit cap (that is, 
    $\capCG = \false$).}
  \item \revised{\emph{TR-Newton-CG (no reg.)}.  The same as
    \emph{TR-Newton-CG}, except that the regularization term
    involving $\epsH$ is omitted from the subproblem objective.  This
    method is the most similar to traditional trust-region Newton-CG
    with Steihaug-Toint stopping rules for the CG routine.  The only
    differences with the latter approach is that, for consistency
    with the other methods, we use the stopping test on line~\ref{step:res}
    of \cref{alg:etcg}, and still use 
    the MEO (\cref{alg:meo}) to ensure convergence to an
    $(\epsg,\epsH)$-stationary point (with high probability).}
  \item \emph{TRACE}.  The trust-region algorithm with
    guaranteed optimal complexity proposed and analyzed in
    \cite{FECurtis_DPRobinson_MSamadi_2017}.
\end{itemize}

\revised{For the six instances of our algorithms, we used a  
regularization of $2\epsH$ while computing the Newton-type step. We set 
$\gamma_1=\gamma_2^{-1}=0.5$, $\psi=0.75$, 
$\eta=0.1$, $\zeta=0.25$, $\delta_{\max}=10^{20}$ and $\delta_0=10$.}


Our experiments show that the empirical performance of these methods
is similar in terms of the number of iterations, function evaluations,
and gradient evaluations required to locate an
$(\epsg,\epsH)$-stationary point.
\revised{In particular, the performance of the Newton-CG variants in 
terms of iterations, function and gradient evaluations is almost 
unaffected by the presence of a cap.}
A second observation is that the regularization term in the
trust-region subproblem objective in \cref{alg:inexact2}, which is
required to ensure optimal iteration and operation complexity
properties for this method, has a noticeable effect on practical
performance, \revised{mostly in terms of Hessian-vector products. We
  discuss this effect further below.}

We tested the algorithms using problems from the CUTEst test
collection \cite{NIMGould_DOrban_PhLToint_2015}. \revised{Many
  problems in this benchmark come with different size options. If the
  default size (according to the sizes that come with the
  distribution, downloaded July 1, 2020) was in the range [100,1000],
  then we used the default size.  Otherwise, we choose the size
  closest to the range [100,1000] from the default value. This
  resulted in a test set of 233 problems.  In order to focus on the
  results of experiments for the larger problems in the set, the
  following discussion and presentation of results considers only the
  problems with $n \geq 100$; this is a test set of 109 problems.}

Figure~\ref{fig:profiles} shows performance profiles for various
metrics \cite{DolaMore02}.  The horizontal axis is capped at $\tau=10$
in order to distinguish the performance of the methods more clearly.
We considered two termination tolerances.  In the first set of
experiments, corresponding to the left column of plots in
Figure~\ref{fig:profiles}, termination was declared when the algorithm
encountered a $(10^{-5},10^{-5/2})$-stationary point.  In the second
set of experiments (the right column of plots in
Figure~\ref{fig:profiles}) we terminate at
$(10^{-5},10^{-5})$-stationary points. In both sets of runs, we
imposed an iteration limit of $10^4$.  For the trust-region Newton-CG
methods, we also imposed an overall Hessian-vector product limit of
$10^4n$: A run was declared to be unsuccessful if this limit is
reached without a stationary point of the specified precision being
found. Although not evident from the performance profiles due to the
cap on $\tau$, all algorithms solved \revised{at least 101 test 
problems out of 109} for both stationarity tolerances, a reliability 
of about \revised{93\%}.

\begin{figure}
  \centering
  \includegraphics[width=.46\textwidth]{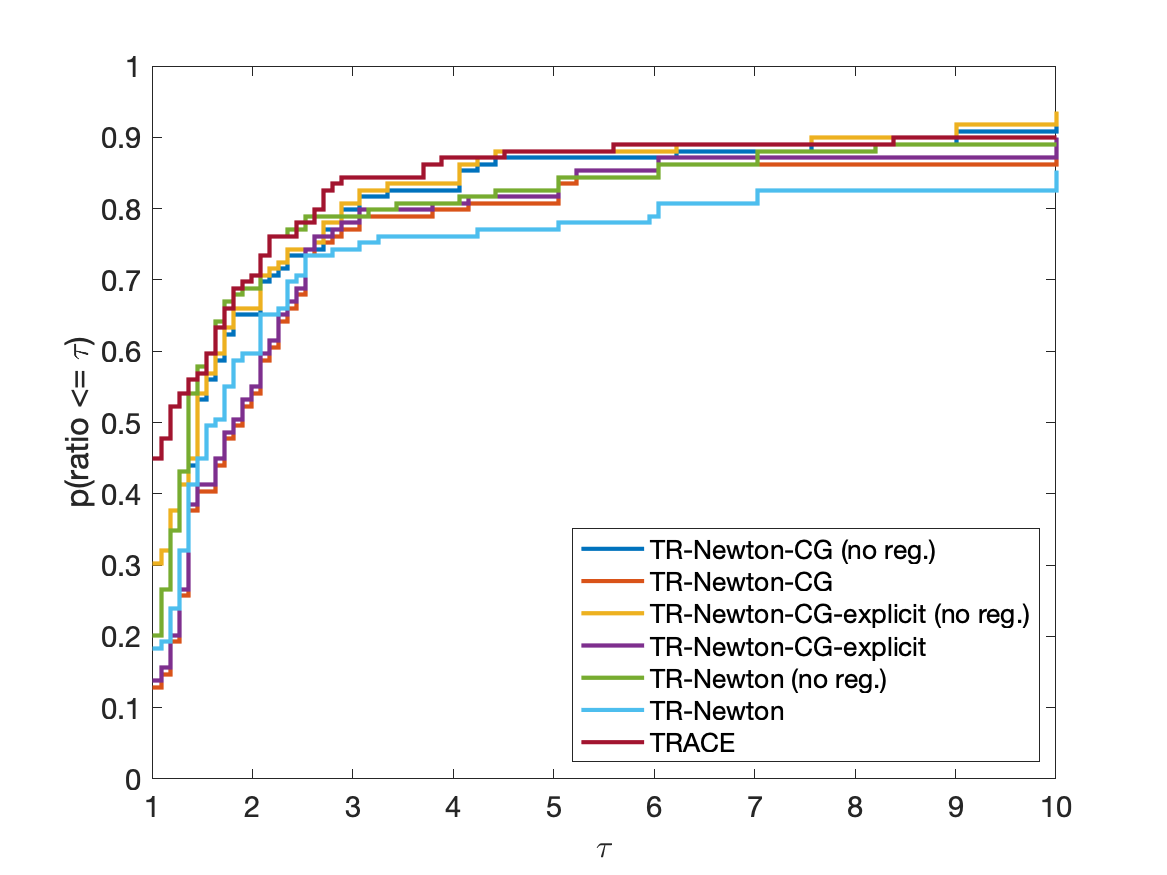}\ 
  \includegraphics[width=.46\textwidth]{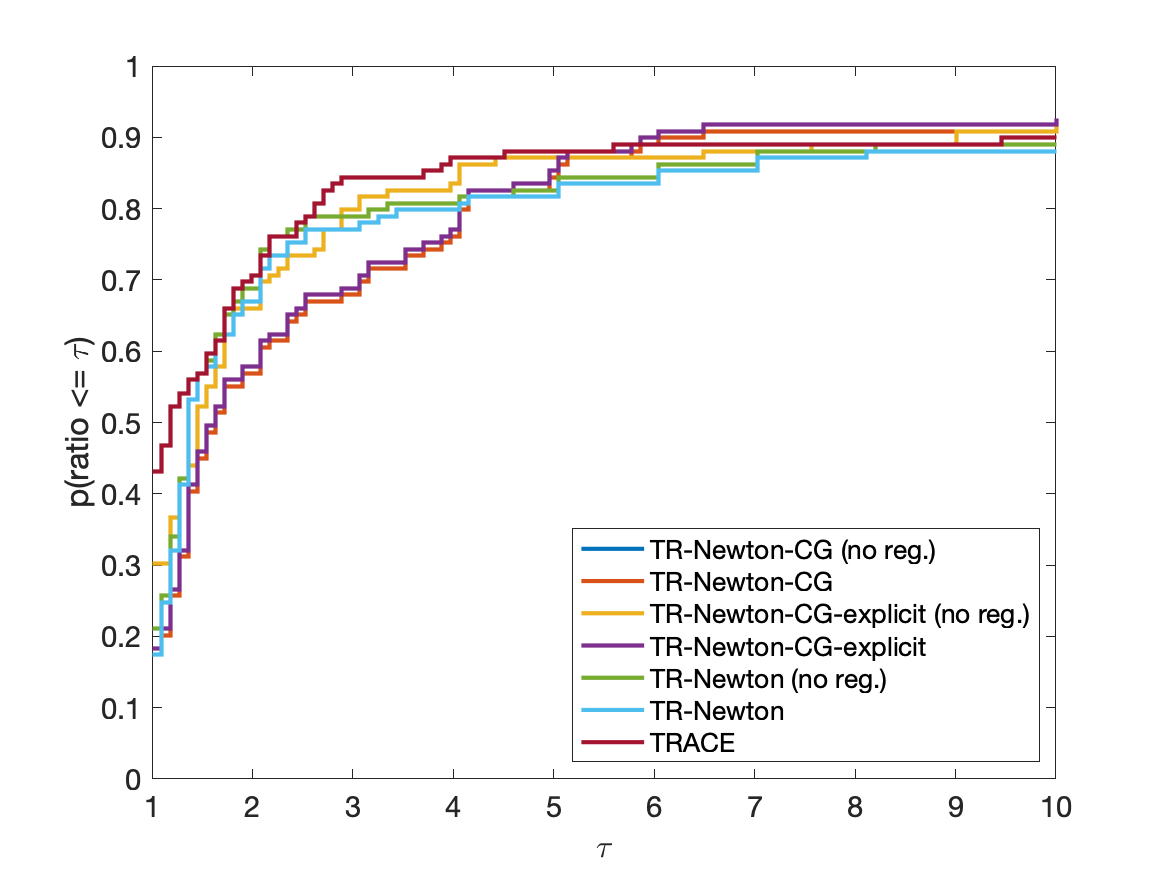} \\
  \includegraphics[width=.46\textwidth]{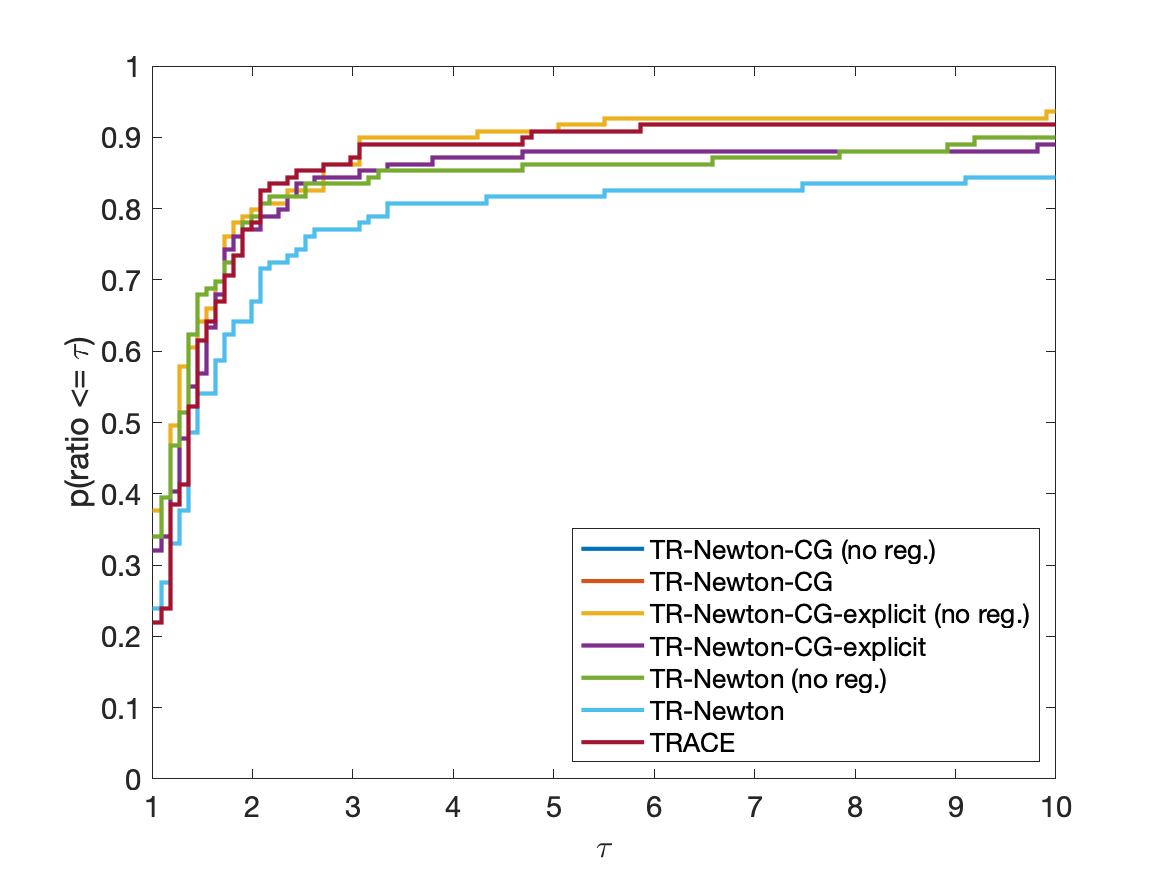}\ 
  \includegraphics[width=.46\textwidth]{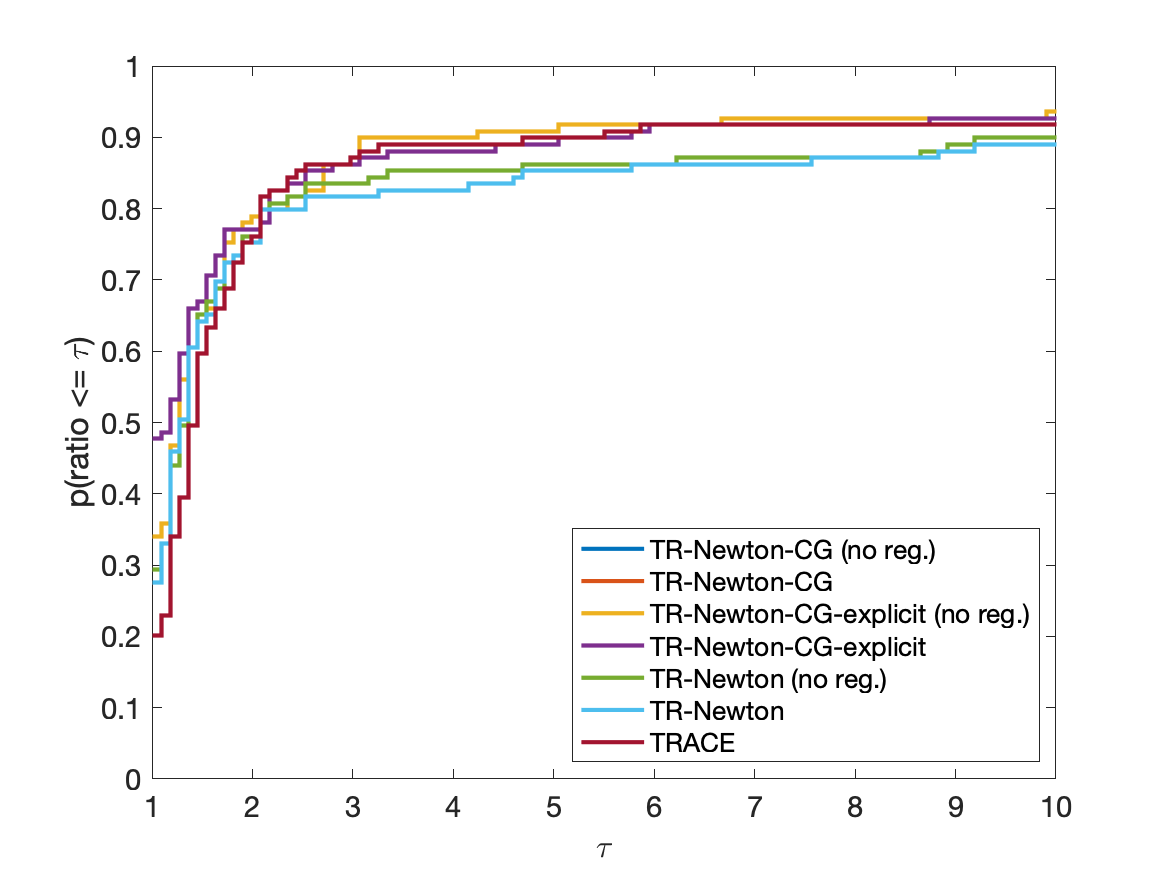} \\
  \includegraphics[width=.46\textwidth]{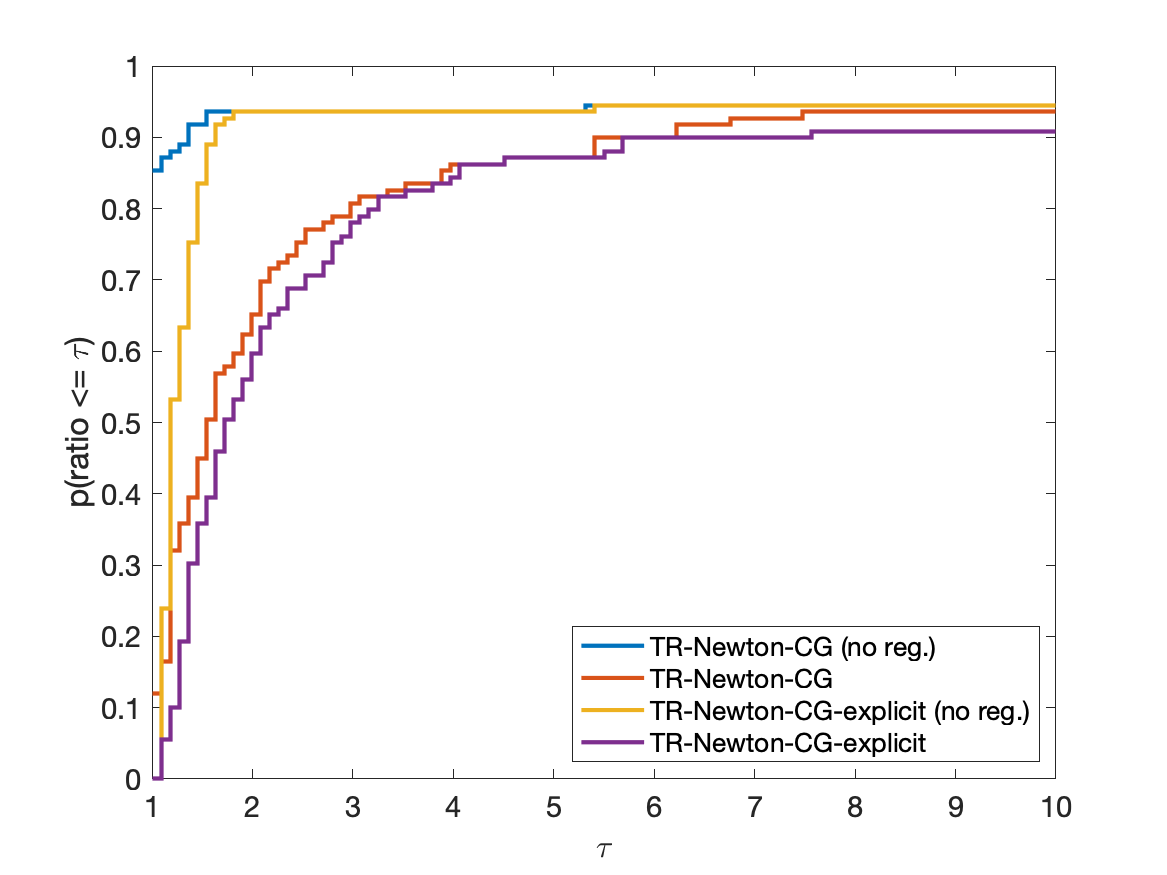}\ 
  \includegraphics[width=.46\textwidth]{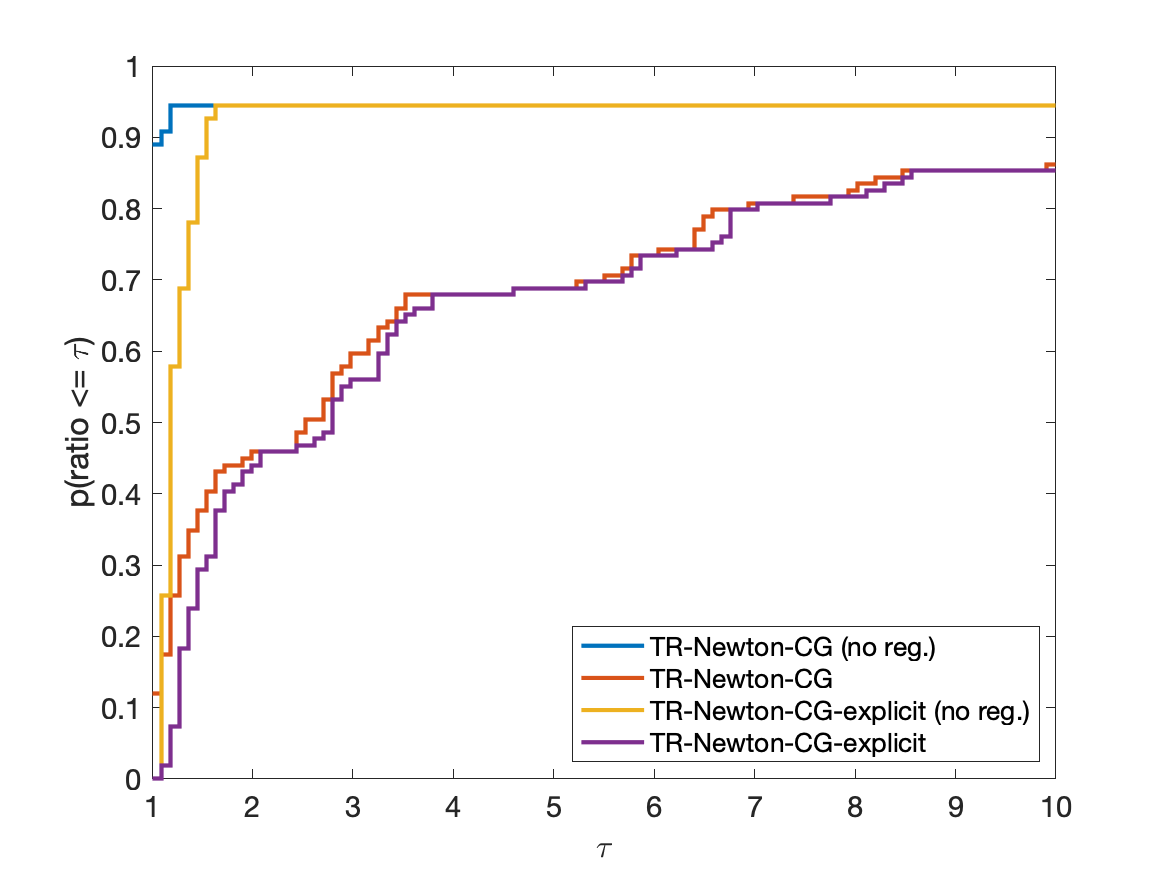} \\
  \caption{Performance profiles for iterations (top), gradient
    evaluations (middle), and Hessian-vector products (bottom). \revised{A
    termination tolerance of $(\epsg,\epsH)=(10^{-5},10^{-5/2})$ is
    used for the left column, and a termination tolerance of $(\epsg,\epsH) = (10^{-5},10^{-5})$ is used for the right column.}}
  \label{fig:profiles}
\end{figure}

Figure~\ref{fig:profiles} shows the performance of all algorithms to be
similar in terms of required iterations and gradient evaluations.
We do however see significant differences in the number of
Hessian-vector products required for the \revised{four} TR-Newton-CG
methods.  The variant\revised{s} with no regularization term in the
subproblems outperform the others in this respect\revised{; recall
  that these variants do not possess optimal complexity
  guarantees}. The practical significance of this difference in
performance depends on the cost of computing a gradient relative to
the cost of a Hessian-vector product.  If gradient evaluations are
significantly more expensive, our results suggest no substantial
difference in computation time between the \revised{four} Newton-CG
methods. On the other hand, if Hessian-vector products are expensive
relative to gradients, there may be a significant increase in run time
as a result of including the regularization term.  \revised{We remark
  that the vast majority of the Hessian-vector products in the
  Newton-CG variants were computed in \cref{alg:etcg}: on all problems
  but three, \cref{alg:meo} was only called at the last iteration to
  assess termination.}


\section{Conclusion} 
\label{sec:conc}
%

We have established that, with a few critical modifications, the
popular trust-region Newton-CG method can be equipped with
second-order complexity guarantees that match the best known bounds
for second-order methods for solving smooth nonconvex optimization
problems. We derived iteration complexity results for both exact and
inexact variants of the approach, and for the inexact variant we
leveraged iterative linear algebra techniques to obtain strong
operation complexity guarantees (in terms of gradient computations and
Hessian-vector products) that again match the best known methods in
the literature.  Finally, we showed that the practical effects of
including these modifications \revised{can be} relatively minor.



Our results could be modified to obtain alternative complexity results
for approximate $\epsg$-stationary points. For instance, we could
modify Algorithm~\ref{alg:etcg} by monitoring the decrease rate of the
residual norm in a way that the number of CG iterations is subject to
an \emph{implicit cap}~\cite{CWRoyer_MONeill_SJWright_2019}, in place
of the explicit cap used here (when $\capCG = \true$). With
appropriate modifications in Algorithm~\ref{alg:inexact2}, and under
the assumptions of \cref{thm:inexact2.wccHvprod}, one could establish
a deterministic operation complexity bound of $\tcO(\epsg^{-7/4})$ for
reaching an $\epsg$-stationary point.  
\revised{However, the resulting method is significantly more delicate to 
implement, and must be paired with Algorithm~\ref{alg:meo} to be endowed 
with a second-order complexity analysis.}


\section*{Acknowledgments}
\revisedbis{We thank Yue Xie for providing most of the proof of
  \cref{theo:inexact2.wccorder}. We are grateful to the editor and two
  anonymous referees, whose comments led to significant improvements
  to the paper.}

\bibliographystyle{siam} \bibliography{refs-newtontr}

\end{document}